\documentclass[3p]{elsarticle}
\usepackage{amsmath,amsthm,,amsfonts,verbatim}
\newcommand{\norm}[1]{\left\lVert#1\right\rVert}
\newtheorem{theorem}{Theorem}
\newtheorem{lemma}{Lemma}

\newtheorem{remark}{Remark}
\usepackage[ruled]{algorithm2e}
\usepackage{algorithmic}
\usepackage{amsthm}
\usepackage{tabularx}
\usepackage{array}
\usepackage[figuresright]{rotating}
\usepackage{amssymb,verbatim}
\usepackage{booktabs}

\begin{document}

\begin{frontmatter}
\title{Nonlinear Convex Optimization: From Relaxed Proximal Point Algorithm to Prediction Correction Method}
%\title{ Convergence Analysis of Proximal Point based Correction Methods for Nonlinear Convex Optimization}

%\author{ Sai Wang\textsuperscript{1}, Shengjie Xu\textsuperscript{2}, Yi Gong\textsuperscript{1}\footnote{\noindent Corresponding author.\\ \indent \quad E-mail address: gongy@sustech.edu.cn (Y. Gong).}}

\author{ Sai Wang, Yi Gong\footnote{\noindent Corresponding author.\\ \indent \quad E-mail address: gongy@sustech.edu.cn (Y. Gong).}}

\address{ Department of Electrical and Electronic Engineering, Southern University of Science and Technology, Shenzhen, China}
%\\ \textsuperscript{2} Department of Mathematics, Southern University of Science and Technology, Shenzhen, China }}
%\textsuperscript{*}Corresponding author. (e-mail: 11930871@mail.sustech.edu.cn; gongy@sustech.edu.cn).}}
%\thanks{Department of Mathematics, Southern University of Science and Technology, Shenzhen, China.}

\begin{abstract}
Nonlinear convex problems arise in various areas of applied mathematics and engineering. Classical techniques such as the relaxed proximal point algorithm (PPA) and the prediction correction (PC) method were proposed for linearly constrained convex problems. However, these methods have not been investigated for nonlinear constraints. In this paper, we customize the varying proximal matrix to develop the relaxed PPA for nonlinear convex problems. We also extend the PC method to nonlinear convex problems. As both methods are an extension of the PPA-based contraction method, their sequence convergence can be directly established. Moreover, we theoretically demonstrate that both methods can achieve a convergence rate of $O(1/t)$. Numerical results once again support the theoretical analysis.

%Nonlinear convex problems emerge in many areas of applied mathematics and engineering applications. As a classical technique, the relaxed proximal point algorithm (PPA) and the prediction correction (PC) method were proposed for linearly constrained convex problems. However, nonlinear constraints have not been investigated yet. In this paper, the relaxed PPA is first developed by customizing the varying proximal matrix for nonlinear convex problems. The PC method is also extended to nonlinear convex problems. Since both of them are an extended version of the PPA-based contraction method, their convergence can be established by using the unified contraction framework. Numerical results once again support the theoretical analysis.
\end{abstract}

\end{frontmatter}

\section{Introduction} \label{section1}

The proximal point algorithm (PPA) plays a fundamental role in the area of convex programming. The first PPA was proposed by Martinet \cite{A1} for convex problems and extended to include metric proximal regularization  \cite{A2,A3,A4,R2}. A linearly inequality-constrained convex problem is written as
\begin{equation}
\min_{\boldsymbol{x}}\left\{ f(\boldsymbol{x})\mid \boldsymbol{A}\boldsymbol{x}\leq\boldsymbol{b}, \boldsymbol{x}\in \mathcal{X}\right\},\label{eq1a}
\end{equation}
where $ f(\boldsymbol{x}): \mathbb{R}^{n}\rightarrow \mathbb{R}$ is a closed convex function (not necessarily smooth), $\boldsymbol{A}\in \mathbb{R}^{m\times n}$, $\boldsymbol{b}\in \mathbb{R}^{m}$ and $\mathcal{X}\subset \mathbb{R}^{n}$ is a closed convex set. By attaching a Lagrange multiplier $\boldsymbol{\lambda}\in \mathbb{R}^{m}$ to the linear constraint, the Lagrange function of (\ref{eq1a}) is
\begin{equation}
  \mathcal{L}'(\boldsymbol{x},\boldsymbol{\lambda})= f(\boldsymbol{x})+\boldsymbol{\lambda}^{T}(\boldsymbol{A}\boldsymbol{x}-\boldsymbol{b}).\ \label{eq1b}
\end{equation}
The primal-dual hybrid gradient (PDHG) algorithm was proposed for saddle point problems in \cite{PD}. Since the Lagrange function (\ref{eq1b}) can be viewed as a saddle point problem, it is natural to employ the PDHG algorithm to solve it. The iterative steps of the algorithm are as follows.
\begin{equation} \left\{ \begin{array}{lrl}
%\begin{align}
\boldsymbol{x}^{k+1}=\arg\min\{\mathcal{L}'(\boldsymbol{x},\boldsymbol{\lambda}^{k})+\frac{r}{2}\norm{\boldsymbol{x}-\boldsymbol{x}^{k}}^{2}\mid \boldsymbol{x}\in \mathcal{X}\},\\
\boldsymbol{\lambda}^{k+1}=\arg\max\{\mathcal{L}'(\boldsymbol{x}^{k+1},\boldsymbol{\lambda})-\frac{s}{2}\Vert\boldsymbol{\lambda}-\boldsymbol{\lambda}^{k}\Vert^{2}\}.
%\end{align}
\end{array}\right.\label{eq1c}
\end{equation}
where $r>0$ and $s>0$ are two regularization factors. Unfortunately, it may fail to identify the optimal solution of linear programming due to the unbalanced alternating iteration \cite{E4}. The customized PPA has been proposed as a remedy to overcome this situation \cite{E0}. By customizing the proximal matrix (symmetric and positive-definite), its convergence can be obtained directly. Specifically, the iterative scheme of the customized PPA for solving (\ref{eq1a}) is given as follows.
\begin{equation} \left\{ \begin{array}{lrl}
%\begin{align}
\boldsymbol{x}^{k+1}=\arg\min\{\mathcal{L}'(\boldsymbol{x},\boldsymbol{\lambda}^{k})+\frac{r}{2}\norm{\boldsymbol{x}-\boldsymbol{x}^{k}}^{2}\mid \boldsymbol{x}\in \mathcal{X}\},\\
\boldsymbol{\lambda}^{k+1}=\arg\max\{\mathcal{L}'(2\boldsymbol{x}^{k+1}-\boldsymbol{x}^{k},\boldsymbol{\lambda})-\frac{s}{2}\Vert\boldsymbol{\lambda}-\boldsymbol{\lambda}^{k}\Vert^{2}\mid \boldsymbol{\lambda}\in \mathbb{R}^{m}_{+}\}.
%\end{align}
\end{array}\right.\label{eq1d}
\end{equation}
 Over the past few decades, the customized PPA for linearly constrained convex optimization has been studied by many researchers. For instance, the authors in \cite{E1} introduced a parameter to customize the proximal matrix for solving $\ell_{1}$-minimization problems arising from compressed sensing. A relaxed customized PPA was developed for linearly constrained separable convex programming problem \cite{E2}. In \cite{E4}, a simple convergence analysis of PPA-like contraction methods is established by leveraging the variational inequality framework. Additionally, the authors in \cite{E1,E4} the PPA-based prediction correction (PC) method was proposed. Unlike the customized PPA, the PC method generates predictive variables using PPA and corrects them using a correction matrix.  Convergence analysis in \cite{E4} shows that the PC method is convergent for linear constraints. In \cite{E5}, an alternative extrapolation scheme of PDHG was proposed for solving nonlinear saddle-point problems. Furthermore, the authors in \cite{E7} proposed an inexact customized PPA framework for two-block separable convex optimization problems based on the error criterion. Without the relaxation step, a new customized PPA was proposed for linearly constrained convex problems in \cite{E3}.  

However, the aforementioned work only focuses on convex problems with linear constraints, and the PPA-based unified framework has not been extended to nonlinearly constrained situations.  In this paper, we aim to address a more general scenario by considering a nonlinear convex optimization problem with inequality constraints:
\begin{equation}
\begin{aligned}
 \min\left\{ f(\boldsymbol{x}) \mid   \phi_{i}(\boldsymbol{x}) \leq 0, \ \boldsymbol{x}\in \mathcal{X}, \ i=1,\cdots,m\right\},
\end{aligned}\label{eq2}
\end{equation}
where $ f(\boldsymbol{x}): \mathbb{R}^{n}\rightarrow \mathbb{R}$ and $\{ \phi_{i}(\boldsymbol{x}): \mathbb{R}^{n}\rightarrow \mathbb{R}, \forall i\}$ are convex functions and $\{ \phi_{i}(\boldsymbol{x}), \forall i\}$ are continuously differentiable. $\mathcal{X}$ is a nonempty closed convex set in $\mathbb{R}^{n}$. Note that $ f(\boldsymbol{x})$ is not necessarily differentiable. The Lagrangian function of problem (\ref{eq2}) is written as 
\begin{equation}
  \mathcal{L}''(\boldsymbol{x},\boldsymbol{\lambda})= f(\boldsymbol{x})+\boldsymbol{\lambda}^{T}  \Phi (\boldsymbol{x}),\ \label{eq19}
\end{equation}
where $\Phi(\boldsymbol{x})=[\phi_{1}(\boldsymbol{x}),\cdots,\phi_{m}(\boldsymbol{x})]$ and the dual variable $\boldsymbol{\lambda}\in \mathbb{R}^{m\times 1}$ is defined in  $\mathcal{Z}: \mathbb{R}^{m}_{+}$. %In problem (\ref{eq2}), all constraints form the vector function $\Phi (\boldsymbol{x})=[ \phi_{1},\cdots, \phi_{m}]^{T}\in \mathbb{R}^{m\times 1}$ and we have $\boldsymbol{\lambda}^{T}  \Phi (\boldsymbol{x})=\sum_{i=1}^{m} \lambda_{i} \phi_{i}(\boldsymbol{x}).$ %The Lagrangian function (\ref{eq19}) is a saddle-point problem that is convex in $\mathcal{X}$ and concave in $ y $.
A pair of $(\boldsymbol{x}^{*},\boldsymbol{\lambda}^{*})$ is called the saddle point of Lagrangian function (\ref{eq19}), i.e.,  
\begin{equation}
 \left\{ \begin{array}{rrl}
     \boldsymbol{x}^{*}\in\arg\min\{\mathcal{L}''(\boldsymbol{x},\boldsymbol{\lambda}^{*})\mid \boldsymbol{x}\in \mathcal{X}\},\\ 
     \boldsymbol{\lambda}^{*}\in\arg\max\{\mathcal{L}''(\boldsymbol{x}^{*},\boldsymbol{\lambda})\mid \boldsymbol{\lambda}\in  \mathcal{Z}\}.
        \end{array}\right.\label{eq20}
\end{equation}
They satisfy the inequality: $\mathcal{L}''_{\boldsymbol{\lambda}\in \mathcal{Z}}(\boldsymbol{x}^{*},\boldsymbol{\lambda})\leq\mathcal{L}''(\boldsymbol{x}^{*},\boldsymbol{\lambda}^{*})\leq\mathcal{L}''_{\boldsymbol{x}\in \mathcal{X}}(\boldsymbol{x},\boldsymbol{\lambda}^{*}).$ In this paper, the saddle point of  Lagrangian function (\ref{eq19})  is assumed to always exist.

The rest of this paper is organized as follows. Section \ref{section2} summarizes some preliminaries for further analysis. In Section \ref{section3} and Section \ref{section4}, we design the relaxed PPA and the PC method, respectively. In Section \ref{section5}, some numerical experiments are reported. Finally, we conclude this paper in Section \ref{section6}.

\section{Preliminaries}\label{section2}
In this section, we introduce some preliminaries that are useful for further analysis. In particular, we review some basic knowledge of variational inequality. The derivative of a scalar $f(\boldsymbol{x})$ by a vector $\boldsymbol{x}=[x_{1},\cdots,x_{n}]^{T}$, is written (in numerator layout notation) as
\begin{equation}
\nabla f(\boldsymbol{x})=\left( \begin{array}{c} \frac{\partial f}{\partial x_{1}} \\ \vdots\\ \frac{\partial f}{\partial x_{n}}\end{array} \right) \mbox{ and let } \mathcal{D}f(\boldsymbol{x})=[\nabla f(\boldsymbol{x})]^{T}.
\nonumber
\end{equation}
The derivative of a vector function (a vector whose components are functions) $\Phi(\boldsymbol{x})=[\phi_{1}(\boldsymbol{x}),\cdots,\phi_{m}(\boldsymbol{x})]^{T}$, with respect to an input vector, $\boldsymbol{x}=[x_{1},\cdots,x_{n}]^{T}$, is written (in numerator layout notation) as
\begin{equation}
\mathcal{D} \Phi(\boldsymbol{x})=\left( \begin{array}{c} \mathcal{D}\phi_{1}(\boldsymbol{x}) \\ \vdots\\ \mathcal{D}\phi_{m}(\boldsymbol{x})\end{array} \right)=\left( \begin{array}{ccc} \frac{\partial \phi_{1}}{\partial x_{1}} &\cdots&\frac{\partial \phi_{1}}{\partial x_{n}} \\ \vdots&\ddots&\vdots\\ \frac{\partial \phi_{m}}{\partial x_{1}}  &\cdots&\frac{\partial \phi_{m}}{\partial x_{n}} \end{array} \right)\nonumber
\end{equation}

\subsection{Differential Convex Optimization}
\begin{lemma}\label{lemma1}
Let $\mathcal{X}\subset \mathbb{R}^{n}$ be a closed convex set and the parameter $ \lambda \ge 0$ hold. Let $ f(\boldsymbol{x})$ and $  g(\boldsymbol{x})$ be convex functions and $g(\boldsymbol{x})$ be differentiable. Assume that the solution set of the minimization problem $\min\{ f(\boldsymbol{x})+ \lambda g(\boldsymbol{x})\mid \boldsymbol{x}\in \mathcal{X}\}$ is nonempty. Then the vector $\boldsymbol{x}^{*}$ is an optimal solutions, i.e., 
\begin{equation}
\boldsymbol{x}^{*}\in \arg \min \{ f(\boldsymbol{x})+ \lambda  g(\boldsymbol{x})\mid\boldsymbol{x}\in \mathcal{X}\}, \label{eq3}
\end{equation}
if and only if the following variational inequality holds
\begin{equation}
\boldsymbol{x}^{*}\in \mathcal{X}, \quad f(\boldsymbol{x})- f(\boldsymbol{x}^{*})+ \lambda (\boldsymbol{x}-\boldsymbol{x}^{*})^{T}\nabla  g(\boldsymbol{x}^{*})\ge0,\quad\forall \ \boldsymbol{x}\in \mathcal{X}.\label{eq4}
\end{equation}
\end{lemma}
\begin{proof}
Sufficiency: If formula (\ref{eq3}) holds, for any $\boldsymbol{x}\in \mathcal{X}$, we have 
\begin{equation}
\frac{ f(\boldsymbol{x}_{\alpha})- f(\boldsymbol{x}^{*})}{\alpha}+ \lambda\frac{  g(\boldsymbol{x}_{\alpha})-  g(\boldsymbol{x}^{*})}{\alpha}\ge0,\label{eq5}
\end{equation}
where $\boldsymbol{x}_{\alpha}=(1-\alpha)\boldsymbol{x}^{*}+\alpha\boldsymbol{x}, \forall \ \alpha \in (0,1]$. Since $ f(\boldsymbol{x})$ is convex, it follows that 
\begin{equation}
\begin{aligned}
 f(\boldsymbol{x}_{\alpha})\leq& (1-\alpha) f(\boldsymbol{x}^{*})+\alpha f(\boldsymbol{x})\label{eq6 }
% f(\boldsymbol{x}_{\alpha})\leq&  f(\boldsymbol{x}^{*})+\alpha[ f(\boldsymbol{x})- f(\boldsymbol{x}^{*})]\\
\end{aligned}
\end{equation}
and thus 
\begin{equation}
 f(\boldsymbol{x})- f(\boldsymbol{x}^{*})\ge\frac{ f(\boldsymbol{x}_{\alpha})- f(\boldsymbol{x}^{*})}{\alpha}.\label{eq7}
\end{equation}
Adding (\ref{eq5}) to (\ref{eq7}), we get 
\begin{equation}
 f(\boldsymbol{x})- f(\boldsymbol{x}^{*})+ \lambda \frac{  g(\boldsymbol{x}_{\alpha})-  g(\boldsymbol{x}^{*})}{\alpha}\ge0.\label{eq8 }
\end{equation}
Taking $\alpha\to 0_{+}$, we have 
\begin{equation}
\begin{aligned}
 f(\boldsymbol{x})- f(\boldsymbol{x}^{*})+ \lambda (\boldsymbol{x}-\boldsymbol{x}^{*})^{T}\nabla  g(\boldsymbol{x}^{*})\ge0, \quad \forall \ \boldsymbol{x}\in \mathcal{X}.\label{eq9 }
\end{aligned}
\end{equation}
Necessity: Since $  g(\boldsymbol{x})$ is convex and $\lambda\ge0$, it follows that 
\begin{equation}
  g(\boldsymbol{x})-  g(\boldsymbol{x}^{*})\ge\frac{  g(\boldsymbol{x}_{\alpha})-  g(\boldsymbol{x}^{*})}{\alpha}.\label{eq10}
\end{equation}
Taking $\alpha\to 0_{+}$, we obtain 
\begin{equation}
  \lambda[g(\boldsymbol{x})-  g(\boldsymbol{x}^{*})]\ge \lambda(\boldsymbol{x}-\boldsymbol{x}^{*})^{T}\nabla  g(\boldsymbol{x}^{*}).\label{eq11}
\end{equation}
When inequality (\ref{eq4}) holds, by adding (\ref{eq11}) to (\ref{eq4}), we have
\begin{equation}
 f(\boldsymbol{x})- f(\boldsymbol{x}^{*})+ \lambda [ g(\boldsymbol{x})-  g(\boldsymbol{x}^{*})]\ge0, \quad\forall \ \boldsymbol{x}\in \mathcal{X}. \label{eq12}
\end{equation}
Thus, this lemma holds.
\end{proof}
\begin{remark}
By Lemma \ref{lemma1}, one of the convex functions is proper and differentiable. Another one can be convex but not necessarily differentiable. 
\end{remark}

\subsection{Monotone Operator}
\begin{lemma}
Let  $\mathcal{X}\subset \mathbb{R}^{n}$ and $ \mathcal{Z} \subset \mathbb{R}^{m}_{+}$ be two closed convex sets. The vector variables $\boldsymbol{x}\in \mathcal{X}$ and $\boldsymbol{\lambda}\in  \mathcal{Z}$ form the new one $\boldsymbol{w}=[\boldsymbol{x}^{T}, \boldsymbol{\lambda}^{T}]^{T}\in\mathbb{R}^{(n+m)\times 1}$. 
For the given $m$ convex problems $\phi_{1}(\boldsymbol{x}),\cdots, \phi_{m}(\boldsymbol{x})$, they form the vector function $  \Phi (\boldsymbol{x})=[ \phi_{1},\cdots, \phi_{m}]^{T}\in\mathbb{R}^{m\times 1}$. The monotone operator $\boldsymbol{\Gamma}(\boldsymbol{w})=[\boldsymbol{\lambda}^{T}\mathcal{D}  \Phi (\boldsymbol{x}), -  \Phi (\boldsymbol{x})^{T}]^{T}\in\mathbb{R}^{(n+m)\times 1}$ satisfies
\begin{equation}
(\boldsymbol{w}-\tilde{\boldsymbol{w}})^{T}[\boldsymbol{\Gamma}(\boldsymbol{w})-\boldsymbol{\Gamma}(\tilde{\boldsymbol{w}})]\ge0, \quad\forall \ \boldsymbol{w}, \tilde{\boldsymbol{w}}\in \mathcal{X}\times \mathcal{Z}.\label{eq13}
\end{equation}\label{lemma2}
\end{lemma}\vspace{-1.0cm}
\begin{proof}
For any $i\in \{1,\cdots,m\}$, suppose $ \phi_{i}(\boldsymbol{x})$ is differentiable on the neighborhood $\mathsf{N}_{r}(\boldsymbol{x}_{0})$. Then for any $\boldsymbol{x}\in\mathsf{N}_{r}(\boldsymbol{x}_{0})$, there exists $\bar{\boldsymbol{x}}$ on the line segment connecting $\boldsymbol{x}$ and $\boldsymbol{x}_{0}$ such that
\begin{equation}
 \phi_{i}(\boldsymbol{x})= \phi_{i}(\boldsymbol{x}_{0}) + \nabla \phi_{i}(\boldsymbol{x}_{0})(\boldsymbol{x}-\boldsymbol{x}_{0})+ \frac{1}{2}(\boldsymbol{x}-\boldsymbol{x}_{0})^{T}\nabla^{2} \phi_{i}(\bar{\boldsymbol{x}})(\boldsymbol{x}-\boldsymbol{x}_{0}).\label{eq14 }
\end{equation}
Since $ \phi_{i}(\boldsymbol{x})$ is convex, we have
\begin{equation}\theta_{i}(\boldsymbol{x})= \phi_{i}(\boldsymbol{x})- \phi_{i}(\boldsymbol{x}_{0}) - \nabla \phi_{i}(\boldsymbol{x}_{0})(\boldsymbol{x}-\boldsymbol{x}_{0})=\frac{1}{2}(\boldsymbol{x}-\boldsymbol{x}_{0})^{T}\nabla^{2} \phi_{i}(\bar{\boldsymbol{x}})(\boldsymbol{x}-\boldsymbol{x}_{0})\ge 0.\label{eq15}
\end{equation}
When $ \phi_{i}(\boldsymbol{x})$ is a linear function, $\nabla^{2} \phi_{i}(\bar{\boldsymbol{x}})=0$ and $\theta_{i}(\boldsymbol{x})=0$. When $ \phi_{i}(\boldsymbol{x})$ is twice differentiable, the hessian matrix $\nabla^{2} \phi_{i}(\bar{\boldsymbol{x}})\succeq0$ is positive semi-definite and $\theta_{i}(\boldsymbol{x})\ge0$. Thus, inequality (\ref{eq15}) holds.
\begin{equation}
\begin{aligned}
(\boldsymbol{w}-\tilde{\boldsymbol{w}})^{T}[\boldsymbol{\Gamma}(\boldsymbol{w})-\boldsymbol{\Gamma}(\tilde{\boldsymbol{w}})]
=&\left( \begin{array}{c} \boldsymbol{x} -\tilde{\boldsymbol{x}}\\ \boldsymbol{\lambda}-\tilde{\boldsymbol{\lambda}} \end{array} \right)^{T}\left( \begin{array}{c} \mathcal{D} \Phi (\boldsymbol{x})^{T}\boldsymbol{\lambda} -\mathcal{D} \Phi (\tilde{\boldsymbol{x}})^{T}\tilde{\boldsymbol{\lambda}}\\
      -  \Phi (\boldsymbol{x}) +  \Phi (\tilde{\boldsymbol{x}})\end{array} \right)\\
 =&\boldsymbol{\lambda}^{T}[  \Phi (\tilde{\boldsymbol{x}})-  \Phi (\boldsymbol{x})-\mathcal{D} \Phi (\boldsymbol{x})(\tilde{\boldsymbol{x}}-\boldsymbol{x} ) ]+\tilde{\boldsymbol{\lambda}}^{k} [  \Phi (\boldsymbol{x}) -  \Phi (\tilde{\boldsymbol{x}})-\mathcal{D} \Phi (\tilde{\boldsymbol{x}})(\boldsymbol{x} -\tilde{\boldsymbol{x}})]\\
 =&\boldsymbol{\lambda}^{T}\boldsymbol{\Theta}(\tilde{\boldsymbol{x}})+\tilde{\boldsymbol{\lambda}}^{k}\boldsymbol{\Theta}(\boldsymbol{x}),
\end{aligned}\label{eq16 }
\end{equation}
where $\boldsymbol{\Theta}(\cdot)=[\theta_{1},\cdots,\theta_{m}]^{T}\in \mathbb{R}^{m\times 1}$. Since $\boldsymbol{\Theta}\ge0$, $\boldsymbol{\lambda}\ge0$ and $\tilde{\boldsymbol{\lambda}}\ge0$, we have 
\begin{equation}
(\boldsymbol{w}-\tilde{\boldsymbol{w}})^{T}[\boldsymbol{\Gamma}(\boldsymbol{w})-\boldsymbol{\Gamma}(\tilde{\boldsymbol{w}})]\ge0, \quad\forall \ \boldsymbol{w}, \tilde{\boldsymbol{w}}\in \mathcal{X}\times \mathcal{Z}.\label{eq17 }
\end{equation}
Thus, this lemma holds.
\end{proof}
\begin{remark}
When the variable $\boldsymbol{\lambda}$ is defined on a set of non positive real numbers $\mathbb{R}_{-}^{m}$, the corresponding monotone operator becomes $\boldsymbol{\Gamma}(\boldsymbol{w})=[-\boldsymbol{\lambda}^{T}\mathcal{D} \Phi (\boldsymbol{x}),   \Phi (\boldsymbol{x})^{T}]^{T}$.
\end{remark}
\subsection{Optimal Condition and Variational Inequality for (\ref{eq2})}
In this part, we review the optimal conditions for problem (\ref{eq2}). Then they can be written as a variational inequality. For convex problem (\ref{eq2}),  the Karush-Kuhn-Tucker (KKT) conditions are given as 
\begin{equation}
\begin{aligned}
  \nabla f(\boldsymbol{x}^{*})+ \sum_{i=1}^{m} \lambda _{i}^{*} \nabla\phi_{i}(\boldsymbol{x}^{*})=\ & 0,\\
\lambda _{i}^{*} \phi_{i}(\boldsymbol{x}^{*})=\ &0, \ i=1,\cdots,m,\\
  \phi_{i} (\boldsymbol{x}^{*})\leq\ &0,\ i=1,\cdots,m,\\
   \lambda_{i}^{*}\ge\ &0,\ i=1,\cdots,m,
  \end{aligned} \label{eq18}\end{equation}
where $\boldsymbol{x}^{*}$ and $( \lambda _{1}^{*},\cdots, \lambda _{m}^{*})$ are any primary and dual optimal point for (\ref{eq2}). 
%They satisfy the following inequality.
%\begin{equation}
%  \mathcal{L}_{\boldsymbol{\lambda}\in  \mathcal{Z}}(\boldsymbol{x}^{*},\boldsymbol{\lambda})\leq \mathcal{L}(\boldsymbol{x}^{*},\boldsymbol{\lambda}^{*})\leq \mathcal{L}_{\boldsymbol{x}\in \mathcal{X}}(\boldsymbol{x},\boldsymbol{\lambda}^{*}).\label{eq21 }
%\end{equation}
By Lemma \ref{lemma1}, the saddle point in (\ref{eq20}) satisfies the following variational inequality:
\begin{equation}
 \left\{ \begin{array}{rrl}
     \boldsymbol{x}^{*}\in \mathcal{X},&  f(\boldsymbol{x})- f(\boldsymbol{x}^{*})+(\boldsymbol{x}-\boldsymbol{x}^{*})^{T}\mathcal{D} \Phi (\boldsymbol{x}^{*})^{T}\boldsymbol{\lambda}^{*}\ge0,& \forall \ \boldsymbol{x}\in \mathcal{X},\\ 
     \boldsymbol{\lambda}^{*}\in  \mathcal{Z},& (\boldsymbol{\lambda}-\boldsymbol{\lambda}^{*})^{T} [- \Phi (\boldsymbol{x}^{*})]\ge0,& \forall \ \boldsymbol{\lambda}\in  \mathcal{Z},\\
        \end{array}\right.\label{eq22 }
\end{equation}
where $\mathcal{D} \Phi (\boldsymbol{x})=[\nabla \phi_{1},\cdots,\nabla \phi_{m}]^{T}\in \mathbb{R}^{m\times n}$ and $\mathcal{D} \Phi (\boldsymbol{x})^{T}\boldsymbol{\lambda}=\sum_{i=1}^{m} \lambda_{i}\nabla \phi_{i}(\boldsymbol{x})\in \mathbb{R}^{n\times 1}.$
The optimal condition can be characterized as a monotone variational inequality:
\begin{equation}
 \boldsymbol{w}^{*}\in \Omega, \quad f(\boldsymbol{x})- f(\boldsymbol{x}^{*})+(\boldsymbol{w}-\boldsymbol{w}^{*})^{T}\boldsymbol{\Gamma}(\boldsymbol{w}^{*})\ge0, \quad\forall \ \boldsymbol{w}\in \Omega, \label{eq23}
\end{equation}
where 
\begin{equation}
 \boldsymbol{w}=\left( \begin{array}{c} \boldsymbol{x} \\
      \boldsymbol{\lambda} \end{array} \right), \
 \boldsymbol{\Gamma}(\boldsymbol{w})=\left( \begin{array}{c} \mathcal{D} \Phi (\boldsymbol{x})^{T}\boldsymbol{\lambda} \\
        -\Phi (\boldsymbol{x}) \end{array} \right) \mbox{ and } \Omega=\mathcal{X}\times \mathcal{Z} .\label{eq24 }
\end{equation}
By Lemma \ref{lemma2}, we know that the operator $\boldsymbol{\Gamma}(\boldsymbol{w})$ is monotone.

\section{Relaxed Proximal Point Algorithm}\label{section3}
In this section, we customize the proximal matrix $\boldsymbol{\Sigma}_{k}$ to ensure that it is symmetrical and positive-definite, in accordance with the convergence criterion of the customized PPA. We then use this customized matrix to derive a relaxed PPA.

\subsection{From the Primal-Dual Method to the Relaxed PPA}
Given $(\boldsymbol{x}^{k},\boldsymbol{\lambda}^{k})$, the primal-dual method alternately produces $\tilde{\boldsymbol{x}}^{k}$ and $\tilde{\boldsymbol{\lambda}}^{k}$ by solving the following problems.
\begin{equation} \left\{ \begin{array}{lrl}
%\begin{align}
\tilde{\boldsymbol{x}}^{k}=\arg\min\{\mathcal{L}''(\boldsymbol{x},\boldsymbol{\lambda}^{k})+\frac{r_{k}}{2}\norm{\boldsymbol{x}-\boldsymbol{x}^{k}}^{2}\mid \boldsymbol{x}\in \mathcal{X}\},\\
\tilde{\boldsymbol{\lambda}}^{k}=\arg\max\{\mathcal{L}''(\tilde{\boldsymbol{x}}^{k},\boldsymbol{\lambda})-\frac{s_{k}}{2}\Vert\boldsymbol{\lambda}-\boldsymbol{\lambda}^{k}\Vert^{2}\mid \boldsymbol{\lambda}\in  \mathcal{Z}\}.
%\end{align}
\end{array}\right.\label{eq25}
\end{equation}
By (\ref{eq25}), we obtain the following variational inequality. 
\begin{equation}
 \left\{ \begin{array}{rrl}
     \tilde{\boldsymbol{x}}^{k}\in \mathcal{X},&  f(\boldsymbol{x})- f(\tilde{\boldsymbol{x}}^{k})+(\boldsymbol{x}-\tilde{\boldsymbol{x}}^{k})^{T}[\mathcal{D} \Phi (\tilde{\boldsymbol{x}}^{k})^{T}\boldsymbol{\lambda}^{k}+r_{k}(\tilde{\boldsymbol{x}}^{k}-\boldsymbol{x}^{k})]\ge0,& \forall \ \boldsymbol{x}\in \mathcal{X},\\ 
     \tilde{\boldsymbol{\lambda}}^{k}\in  \mathcal{Z},& (\boldsymbol{\lambda}-\tilde{\boldsymbol{\lambda}}^{k})^{T}[ -\Phi (\tilde{\boldsymbol{x}}^{k})+s_{k}(\tilde{\boldsymbol{\lambda}}^{k}-\boldsymbol{\lambda}^{k})]\ge0,& \forall \ \boldsymbol{\lambda}\in  \mathcal{Z}.\\
        \end{array}\right. \label{eq26 }
\end{equation}
Then we have 
\begin{equation}
   f(\boldsymbol{x})- f(\tilde{\boldsymbol{x}}^{k})+(\boldsymbol{w}- \tilde{\boldsymbol{w}}^{k})^{T}\left\{\boldsymbol{\Gamma}( \tilde{\boldsymbol{w}}^{k})+\left(
\begin{array}{r}
r_{k}(\tilde{\boldsymbol{x}}^{k}-\boldsymbol{x}^{k}) -\mathcal{D} \Phi (\tilde{\boldsymbol{x}}^{k})^{T}(\tilde{\boldsymbol{\lambda}}^{k}-\boldsymbol{\lambda}^{k}) \\
s_{k}(\tilde{\boldsymbol{\lambda}}^{k}-\boldsymbol{\lambda}^{k})
\end{array}
\right)\right\}\ge0. \quad\forall \ \boldsymbol{w}\in \Omega.\label{eq27 }
\end{equation}
Further, the above variational inequality can be written as
\begin{equation}
 f(\boldsymbol{x})- f( \tilde{\boldsymbol{x}}^{k})+(\boldsymbol{w}- \tilde{\boldsymbol{w}}^{k})^{T}\left\{\boldsymbol{\Gamma}( \tilde{\boldsymbol{w}}^{k})+\boldsymbol{Q}_{k}(\tilde{\boldsymbol{w}}^{k}-\boldsymbol{w}^{k})\right\}\ge0, \quad\forall \ \boldsymbol{w}\in \Omega,\label{eq28}
\end{equation}
where the proximal matrix
\begin{equation} \boldsymbol{Q}_{k}=\left(
\begin{array}{cc}
r_{k}\boldsymbol{I}_{n} & -\mathcal{D} \Phi (\tilde{\boldsymbol{x}}^{k})^{T} \\
\boldsymbol{0}&s_{k}\boldsymbol{I}_{m} 
\end{array}
\right) \mbox{ is not symmetrical.}\nonumber
\end{equation}
To use the customized PPA, the proximal matrix must be symmetric and positive-definite. However, it is clear that $\boldsymbol{Q}_{k}$ does not satisfy this requirement. An example of linear programming in \cite{E4} has shown that the primal-dual method does not converge. Nonetheless, we can modify the non-symmetric matrix $\boldsymbol{Q}_{k}$ into a customized symmetric matrix $\boldsymbol{\Sigma}_{k}$: $\boldsymbol{\Sigma}(\tilde{\boldsymbol{x}}^{k})$ such as
\begin{equation} \boldsymbol{Q}_{k}=\left(
\begin{array}{cc}
r_{k}\boldsymbol{I}_{n} & -\mathcal{D} \Phi (\tilde{\boldsymbol{x}}^{k})^{T} \\
\boldsymbol{0}&s_{k}\boldsymbol{I}_{m} 
\end{array}
\right)\quad\Rightarrow\quad \boldsymbol{\Sigma}_{k}=\left(
\begin{array}{cc}
r_{k}\boldsymbol{I}_{n} & -\mathcal{D} \Phi (\tilde{\boldsymbol{x}}^{k})^{T} \\
-\mathcal{D} \Phi (\tilde{\boldsymbol{x}}^{k})&s_{k}\boldsymbol{I}_{m} 
\end{array}
\right). \label{eq29 }%\nonumber
\end{equation}
Let the singular value decomposition of $\mathcal{D} \Phi (\tilde{\boldsymbol{x}}{}^{k})$ be equal to $\boldsymbol{U}\boldsymbol{\Lambda}\boldsymbol{V}^{T}$, where $\boldsymbol{U}$ and $\boldsymbol{V}$ are unitary matrices and $\boldsymbol{\Lambda}$ is a rectangular diagonal matrix with the singular values lying on the diagonal.
\begin{equation}
\begin{aligned}
r_{k}s_{k}\boldsymbol{I}_{m}-\mathcal{D} \Phi (\tilde{\boldsymbol{x}}{}^{k})\mathcal{D} \Phi (\tilde{\boldsymbol{x}}{}^{k})^{T}
=& r_{k}s_{k}\boldsymbol{I}_{m}-\boldsymbol{U}\boldsymbol{\Lambda}\boldsymbol{V}^{T}\boldsymbol{V}\boldsymbol{\Lambda}\boldsymbol{U}^{T}\\
=& r_{k}s_{k}\boldsymbol{I}_{m}-\boldsymbol{U}\boldsymbol{\Lambda}^{2}\boldsymbol{U}^{T}\\
=& \boldsymbol{U}\left(r_{k}s_{k}\boldsymbol{I}_{m}-\boldsymbol{\Lambda}^{2}\right)\boldsymbol{U}^{T}
\end{aligned}\label{eq30}
\end{equation}
For given $r_{k}>0$ and $s_{k}>0$, the symmetric matrix $\boldsymbol{\Sigma}_{k}$ is positive definite if and only if $r_{k}s_{k}\boldsymbol{I}_{m}-\mathcal{D} \Phi (\tilde{\boldsymbol{x}}{}^{k})\mathcal{D} \Phi (\tilde{\boldsymbol{x}}{}^{k})^{T}\succ0$. According to (\ref{eq30}), to reach this condition, the value of $r_{k}s_{k}$ should be greater than the maximum singular value of $\mathcal{D} \Phi (\tilde{\boldsymbol{x}}{}^{k})\mathcal{D} \Phi (\tilde{\boldsymbol{x}}{}^{k})^{T}$, i.e., $r_{k}s_{k}>\rho(\mathcal{D} \Phi (\tilde{\boldsymbol{x}}{}^{k})\mathcal{D} \Phi (\tilde{\boldsymbol{x}}{}^{k})^{T})$.
Then the variational inequality satisfies this form.
\begin{equation}
  f(\boldsymbol{x})- f( \tilde{\boldsymbol{x}}^{k})+(\boldsymbol{w}- \tilde{\boldsymbol{w}}^{k})^{T}\left\{\boldsymbol{\Gamma}( \tilde{\boldsymbol{w}}^{k})+\boldsymbol{\Sigma}_{k}(\tilde{\boldsymbol{w}}^{k}-\boldsymbol{w}^{k})\right\}\ge0, \quad\forall \ \boldsymbol{w}\in \Omega, \label{eq31}
\end{equation}
For any $ \boldsymbol{w}\in \Omega$, we expend (\ref{eq31}) as following
\begin{equation}
\begin{aligned}
 % f(\boldsymbol{x})- f( \tilde{\boldsymbol{x}}^{k})+(\boldsymbol{w}- \tilde{\boldsymbol{w}}^{k})^{T}\left\{\boldsymbol{\Gamma}( \tilde{\boldsymbol{w}}^{k})+\boldsymbol{\Sigma}_{k}(\tilde{\boldsymbol{w}}^{k}-\boldsymbol{w}^{k})\right\}\ge0\\
   f(\boldsymbol{x})- f( \tilde{\boldsymbol{x}}^{k})+
   \left( \begin{array}{c} \boldsymbol{x}-\tilde{\boldsymbol{x}}^{k} \\
      \boldsymbol{\lambda} -\tilde{\boldsymbol{\lambda}}^{k}\end{array} \right)^{T}\left\{\left( \begin{array}{c} \mathcal{D} \Phi (\tilde{\boldsymbol{x}}^{k})^{T}\tilde{\boldsymbol{\lambda}}^{k}\\
        -\Phi (\tilde{\boldsymbol{x}}^{k}) \end{array} \right)+\left(
\begin{array}{c}
r_{k}(\tilde{\boldsymbol{x}}^{k}-\boldsymbol{x}^{k}) -\mathcal{D} \Phi (\tilde{\boldsymbol{x}}^{k})^{T}(\tilde{\boldsymbol{\lambda}}^{k}-\boldsymbol{\lambda}^{k}) \\
-\mathcal{D} \Phi (\tilde{\boldsymbol{x}}^{k})(\tilde{\boldsymbol{x}}^{k}-\boldsymbol{x}^{k}) +s_{k}(\tilde{\boldsymbol{\lambda}}^{k}-\boldsymbol{\lambda}^{k})
\end{array}
\right)\right\}\ge0.
 \end{aligned}
\end{equation}
By simplifying the above inequality, we obtain 
\begin{equation}
 \left\{ \begin{array}{rrl}
     \tilde{\boldsymbol{x}}^{k}\in \mathcal{X},&  f(\boldsymbol{x})- f(\tilde{\boldsymbol{x}}^{k})+(\boldsymbol{x}-\tilde{\boldsymbol{x}}^{k})^{T}[\mathcal{D} \Phi (\tilde{\boldsymbol{x}}^{k})^{T}\boldsymbol{\lambda}^{k}+r_{k}(\tilde{\boldsymbol{x}}^{k}-\boldsymbol{x}^{k})]\ge0,& \forall \ \boldsymbol{x}\in \mathcal{X},\\ 
     \tilde{\boldsymbol{\lambda}}^{k}\in  \mathcal{Z},& (\boldsymbol{\lambda}-\tilde{\boldsymbol{\lambda}}^{k})^{T}[ -\Phi (\tilde{\boldsymbol{x}}^{k})-\mathcal{D} \Phi (\tilde{\boldsymbol{x}}^{k})(\tilde{\boldsymbol{x}}^{k}-\boldsymbol{x}^{k})+s_{k}(\tilde{\boldsymbol{\lambda}}^{k}-\boldsymbol{\lambda}^{k})]\ge0,& \forall \ \boldsymbol{\lambda}\in  \mathcal{Z}.\\
        \end{array}\right. \label{eq32}
\end{equation}
For the given regularization parameters $(r_{k}, s_{k})$, the customized PPA method solves the following problems.
\begin{equation}
\tilde{\boldsymbol{x}}^{k}=\arg\min\{\mathcal{L}''(\boldsymbol{x},\boldsymbol{\lambda}^{k})+\frac{r_{k}}{2}\norm{\boldsymbol{x}-\boldsymbol{x}^{k}}^{2}\mid \boldsymbol{x}\in \mathcal{X}\}.\label{eq33}
\end{equation}
\begin{equation}\begin{aligned}
\tilde{\boldsymbol{\lambda}}^{k}=&\arg\max\{\mathcal{L}''(\tilde{\boldsymbol{x}}^{k},\boldsymbol{\lambda})+\boldsymbol{\lambda}^{T}\mathcal{D} \Phi (\tilde{\boldsymbol{x}}^{k})(\tilde{\boldsymbol{x}}^{k}-\boldsymbol{x}^{k}) -\frac{s_{k}}{2}\Vert\boldsymbol{\lambda}-\boldsymbol{\lambda}^{k}\Vert^{2}\mid \boldsymbol{\lambda}\in  \mathcal{Z}\}\\
=&P_{\mathcal{Z}}\left(\boldsymbol{\lambda}^{k}+\frac{1}{s_{k}}[ \Phi (\tilde{\boldsymbol{x}}^{k}) +\mathcal{D} \Phi (\tilde{\boldsymbol{x}}^{k})(\tilde{\boldsymbol{x}}^{k}-\boldsymbol{x}^{k})]\right).
\end{aligned}\label{eq34}
\end{equation}
In the relaxed step, we take 
\begin{equation}
\boldsymbol{w}^{k+1}=\boldsymbol{w}^{k}-\gamma(\boldsymbol{w}^{k}-\tilde{\boldsymbol{w}}^{k}), \gamma\in (0, 2).\label{eq35}
\end{equation}
Note that the variational inequality (\ref{eq32}) cannot be written as a saddle point problem because of the presence of  $\boldsymbol{\lambda}^{T}\mathcal{D} \Phi (\tilde{\boldsymbol{x}}^{k})(\tilde{\boldsymbol{x}}^{k}-\boldsymbol{x}^{k})$. %Specifically, the saddle-point problem of (\ref{old37}) can be written as \begin{equation}
%  \mathcal{L}'(\boldsymbol{x},\boldsymbol{\lambda})=\mathcal{L}(\boldsymbol{x},\boldsymbol{\lambda})+\frac{r}{2}\norm{\boldsymbol{x}-\boldsymbol{x}^{k}}^{2}-\frac{s}{2}\Vert\boldsymbol{\lambda}-\boldsymbol{\lambda}^{k}\Vert^{2}.\label{eq66}
%\end{equation}
%While the saddle-point problem of (\ref{old38}) is equal to 
%\begin{equation}
 % \mathcal{L}''(\boldsymbol{x},\boldsymbol{\lambda})=\mathcal{L}(\boldsymbol{x},\boldsymbol{\lambda})-\boldsymbol{\lambda}^{T}\mathcal{D} \Phi (\tilde{\boldsymbol{x}}^{k})(\tilde{\boldsymbol{x}}^{k}-\boldsymbol{x}^{k})+\frac{r}{2}\norm{\boldsymbol{x}-\boldsymbol{x}^{k}}^{2}-\frac{s}{2}\Vert\boldsymbol{\lambda}-\boldsymbol{\lambda}^{k}\Vert^{2}.\label{eq67}
%\end{equation}
\begin{remark}
In practical computation, we choose a step size $\gamma \in [1, 2)$ to increase the speed of convergence. If the parameter $\gamma$ equals 1, then $\boldsymbol{w}^{k+1}=\tilde{\boldsymbol{w}}^{k}$, and the relaxed PPA is transformed into the customized PPA. When all functions in constraints are linear, the proposed PPA is equivalent to the classical customized PPA in \cite{E4}. Therefore, the proposed algorithm is an extended version of PPAs.
\end{remark}

\subsection{Convergence Analysis for the Relaxed PPA}
 Besides the PPA in the Euclidean-norm, we can consider the PPA in $\boldsymbol{\Sigma}_{k}$-norm. During the customized PPA, for the given $\boldsymbol{w}^{k}$, the new iterate $\tilde{\boldsymbol{w}}^{k}$ of the PPA in $\boldsymbol{\Sigma}_{k}$-norm satisfies the following inequality.
\begin{equation}
 \tilde{\boldsymbol{w}}^{k}\in \Omega, \quad f(\boldsymbol{x})- f(\tilde{\boldsymbol{w}}^{k})+(\boldsymbol{w}-\tilde{\boldsymbol{w}}^{k})^{T}[\boldsymbol{\Gamma}(\tilde{\boldsymbol{w}}^{k})-\boldsymbol{\Sigma}_{k}(\tilde{\boldsymbol{w}}^{k}-\boldsymbol{w}^{k})]\ge0,\quad \forall \ \boldsymbol{w}\in \Omega, \label{eq36}
\end{equation}
 where $\tilde{\boldsymbol{w}}^{k}$ is called the proximal point of the $k$-th iteration. Setting $\boldsymbol{w} =\boldsymbol{w}^{*}$ in (\ref{eq36}), we obtain

\begin{equation}
\begin{aligned}
 (\tilde{\boldsymbol{w}}^{k}-\boldsymbol{w}^{*})^{T}\boldsymbol{\Sigma}_{k}(\boldsymbol{w}^{k}-\tilde{\boldsymbol{w}}^{k})\ge&  f(\tilde{\boldsymbol{w}}^{k})- f(\boldsymbol{x}^{*})+(\tilde{\boldsymbol{w}}^{k}-\boldsymbol{w}^{*})^{T}\boldsymbol{\Gamma}(\tilde{\boldsymbol{w}}^{k})\\
 \stackrel{(\ref{eq13})}{\ge}&  f(\tilde{\boldsymbol{w}}^{k})- f(\boldsymbol{x}^{*})+(\tilde{\boldsymbol{w}}^{k}-\boldsymbol{w}^{*})^{T}\boldsymbol{\Gamma}(\boldsymbol{w}^{*})\stackrel{(\ref{eq23})}{\ge}0
 \end{aligned}\label{eq37}
\end{equation}
By (\ref{eq37}), we get 
\begin{equation}
\begin{aligned}
 \left[(\tilde{\boldsymbol{w}}^{k}-\boldsymbol{w}^{k})+(\boldsymbol{w}^{k}-\boldsymbol{w}^{*})\right]^{T}\boldsymbol{\Sigma}_{k}(\boldsymbol{w}^{k}-\tilde{\boldsymbol{w}}^{k})&\ge0\\
 (\boldsymbol{w}^{k}-\boldsymbol{w}^{*})^{T}\boldsymbol{\Sigma}_{k}(\boldsymbol{w}^{k}-\tilde{\boldsymbol{w}}^{k})&\ge(\tilde{\boldsymbol{w}}^{k}-\boldsymbol{w}^{k})^{T}\boldsymbol{\Sigma}_{k}(\tilde{\boldsymbol{w}}^{k}-\boldsymbol{w}^{k})\\
 (\boldsymbol{w}^{k}-\boldsymbol{w}^{*})^{T}\boldsymbol{\Sigma}_{k}(\boldsymbol{w}^{k}-\tilde{\boldsymbol{w}}^{k})&\ge\norm{\tilde{\boldsymbol{w}}{}^{k}-\boldsymbol{w}^{k}}_{\boldsymbol{\Sigma}_{k}}^{2}.
 \end{aligned}\label{eq38}
\end{equation}

\begin{theorem}{(Sequence contraction)}
For the given $\boldsymbol{w}^{k}\in \Omega$, let $\tilde{\boldsymbol{w}}^{k}$ be generated by the customized PPA method and $\boldsymbol{w}^{k+1}=\boldsymbol{w}^{k}-\gamma(\boldsymbol{w}^{k}-\tilde{\boldsymbol{w}}^{k})$. Then the sequence $\boldsymbol{w}^{k}$ converges the optimal solution of problem (\ref{eq2}) and satisfies the following inequality.
\begin{equation}
\norm{\boldsymbol{w}^{k+1}-\boldsymbol{w}^{*}}_{\boldsymbol{\Sigma}_{k}}^{2}\leq\norm{\boldsymbol{w}^{k}-\boldsymbol{w}^{*}}_{\boldsymbol{\Sigma}_{k}}^{2}-\gamma(2-\gamma)\norm{\tilde{\boldsymbol{w}}{}^{k}-\boldsymbol{w}^{k}}_{\boldsymbol{\Sigma}_{k}}^{2},\label{eq39 }
\end{equation}
 where the relaxed parameter $\gamma\in (0, 2)$.
\end{theorem}

\begin{proof}
By (\ref{eq35}) and (\ref{eq38}), we take
\begin{equation}
\begin{aligned}
\norm{\boldsymbol{w}^{k+1}-\boldsymbol{w}^{*}}_{\boldsymbol{\Sigma}_{k}}^{2}
=&\norm{\boldsymbol{w}^{k+1}-\boldsymbol{w}^{k}+\boldsymbol{w}^{k}-\boldsymbol{w}^{*}}_{\boldsymbol{\Sigma}_{k}}^{2}\\
\stackrel{(\ref{eq35})}{=}&\norm{(\boldsymbol{w}^{k}-\boldsymbol{w}^{*})-\gamma(\boldsymbol{w}^{k}-\tilde{\boldsymbol{w}}{}^{k})}_{\boldsymbol{\Sigma}_{k}}^{2}\\
=&\norm{\boldsymbol{w}^{k}-\boldsymbol{w}^{*}}_{\boldsymbol{\Sigma}_{k}}^{2}-2\gamma(\boldsymbol{w}^{k}-\boldsymbol{w}^{*})^{T}\boldsymbol{\Sigma}_{k}(\boldsymbol{w}^{k}-\tilde{\boldsymbol{w}}^{k})+\gamma^{2}\norm{\boldsymbol{w}^{k}-\tilde{\boldsymbol{w}}{}^{k}}_{\boldsymbol{\Sigma}_{k}}^{2}\\
\stackrel{(\ref{eq38})}{\leq}&\norm{\boldsymbol{w}^{k}-\boldsymbol{w}^{*}}_{\boldsymbol{\Sigma}_{k}}^{2}-2\gamma\norm{\tilde{\boldsymbol{w}}{}^{k}-\boldsymbol{w}^{k}}_{\boldsymbol{\Sigma}_{k}}^{2}+\gamma^{2}\norm{\boldsymbol{w}^{k}-\tilde{\boldsymbol{w}}{}^{k}}_{\boldsymbol{\Sigma}_{k}}^{2}\\
=& \norm{\boldsymbol{w}^{k}-\boldsymbol{w}^{*}}_{\boldsymbol{\Sigma}_{k}}^{2}-\gamma(2-\gamma)\norm{\tilde{\boldsymbol{w}}{}^{k}-\boldsymbol{w}^{k}}_{\boldsymbol{\Sigma}_{k}}^{2}.
\end{aligned}\label{eq40 }
\end{equation}
Thus, this theorem holds.
\end{proof}

\begin{lemma}
Let the sequence $\{\boldsymbol{w}^{k},\tilde{\boldsymbol{w}}{}^{k}\}$ be generated by the customized PPA and $\boldsymbol{\Sigma}_{k}$ be the customized proximal matrix. Then we have 
\begin{equation}
f(\boldsymbol{x})-f(\tilde{\boldsymbol{x}}{}^{k})+(\boldsymbol{w}-\tilde{\boldsymbol{w}}{}^{k})\boldsymbol{\Gamma}(\tilde{\boldsymbol{w}}{}^{k})+\frac{1}{2\gamma}(\Vert \boldsymbol{w}-\boldsymbol{w}^{k} \Vert_{\boldsymbol{\Sigma}_{k}}^{2}-\Vert \boldsymbol{w}-\boldsymbol{w}^{k+1} \Vert_{\boldsymbol{\Sigma}_{k}}^{2})\ge (1-\frac{\gamma}{2})\Vert \boldsymbol{w}^{k}-\tilde{\boldsymbol{w}}^{k} \Vert_{\boldsymbol{\Sigma}_{k}}^{2}, \ \forall \ \boldsymbol{w}\in \Omega. \label{eq41}
\end{equation}
\end{lemma}
\begin{proof}
Referring to (\ref{eq35}), the from (\ref{eq36}) follows that 
\begin{equation}
\begin{aligned}
  f(\boldsymbol{x})- f(\tilde{\boldsymbol{w}}^{k})+(\boldsymbol{w}-\tilde{\boldsymbol{w}}^{k})^{T}\boldsymbol{\Gamma}(\tilde{\boldsymbol{w}}^{k})&\ge(\boldsymbol{w}-\tilde{\boldsymbol{w}}^{k})^{T}\boldsymbol{\Sigma}_{k}(\boldsymbol{w}^{k}-\tilde{\boldsymbol{w}}^{k})\\
  &=\frac{1}{\gamma}(\boldsymbol{w}-\tilde{\boldsymbol{w}}^{k})^{T}\boldsymbol{\Sigma}_{k}(\boldsymbol{w}^{k}-\boldsymbol{w}^{k+1}),\quad \forall \ \boldsymbol{w}\in \Omega. \label{eq42}
 \end{aligned}
\end{equation}
By using the following formula 
\begin{equation}
(\boldsymbol{a}-\boldsymbol{b})^{T}\boldsymbol{\Sigma}_{k}(\boldsymbol{c}-\boldsymbol{d})=\frac{1}{2}\left( \norm{\boldsymbol{a}-\boldsymbol{d}}_{\boldsymbol{\Sigma}_{k}}^{2}-\norm{\boldsymbol{a}-\boldsymbol{c}}_{\boldsymbol{\Sigma}_{k}}^{2}\right)+\frac{1}{2}\left( \norm{\boldsymbol{b}-\boldsymbol{c}}_{\boldsymbol{\Sigma}_{k}}^{2}-\norm{\boldsymbol{b}-\boldsymbol{d}}_{\boldsymbol{\Sigma}_{k}}^{2}\right),\label{eq43}
\end{equation}
 we derive that 
\begin{equation}
\begin{aligned}
(\boldsymbol{w}-\tilde{\boldsymbol{w}}^{k})^{T}\boldsymbol{\Sigma}_{k}(\boldsymbol{w}^{k}-\boldsymbol{w}^{k+1})=&\frac{1}{2}\left( \norm{\boldsymbol{w}-\boldsymbol{w}^{k+1}}_{\boldsymbol{\Sigma}_{k}}^{2}-\norm{\boldsymbol{w}-\boldsymbol{w}^{k}}_{\boldsymbol{\Sigma}_{k}}^{2}\right)\\
&+\frac{1}{2}\left( \norm{\tilde{\boldsymbol{w}}{}^{k}-\boldsymbol{w}^{k}}_{\boldsymbol{\Sigma}_{k}}^{2}-\norm{\tilde{\boldsymbol{w}}{}^{k}-\boldsymbol{w}^{k+1}}_{\boldsymbol{\Sigma}_{k}}^{2}\right).\label{eq44}
\end{aligned}
\end{equation}
On the other hand, the following equation holds.
\begin{eqnarray}
\lefteqn{\norm{\tilde{\boldsymbol{w}}{}^{k}-\boldsymbol{w}^{k}}_{\boldsymbol{\Sigma}_{k}}^{2}-\norm{\tilde{\boldsymbol{w}}{}^{k}-\boldsymbol{w}^{k+1}}_{\boldsymbol{\Sigma}_{k}}^{2}}\nonumber\\
&=&\norm{\tilde{\boldsymbol{w}}{}^{k}-\boldsymbol{w}^{k}}_{\boldsymbol{\Sigma}_{k}}^{2}-\norm{(\tilde{\boldsymbol{w}}{}^{k}-\boldsymbol{w}^{k})-(\boldsymbol{w}^{k+1}-\boldsymbol{w}^{k})}_{\boldsymbol{\Sigma}_{k}}^{2}\nonumber\\
&=&\norm{\tilde{\boldsymbol{w}}{}^{k}-\boldsymbol{w}^{k}}_{\boldsymbol{\Sigma}_{k}}^{2}-\norm{(\tilde{\boldsymbol{w}}{}^{k}-\boldsymbol{w}^{k})-\gamma(\tilde{\boldsymbol{w}}{}^{k}-\boldsymbol{w}^{k})}_{\boldsymbol{\Sigma}_{k}}^{2}\\
&=&\norm{\tilde{\boldsymbol{w}}{}^{k}-\boldsymbol{w}^{k}}_{\boldsymbol{\Sigma}_{k}}^{2}-(1-\gamma)^{2}\norm{\tilde{\boldsymbol{w}}{}^{k}-\boldsymbol{w}^{k}}_{\boldsymbol{\Sigma}_{k}}^{2}\nonumber\\
&=&\gamma(2-\gamma)\norm{\tilde{\boldsymbol{w}}{}^{k}-\boldsymbol{w}^{k}}_{\boldsymbol{\Sigma}_{k}}^{2}.\nonumber
\label{eq45 }
\end{eqnarray}
Combining with (\ref{eq44}), we have
\begin{equation}
\begin{aligned}
\frac{1}{\gamma}(\boldsymbol{w}-\tilde{\boldsymbol{w}}^{k})^{T}\boldsymbol{\Sigma}_{k}(\boldsymbol{w}^{k}-\boldsymbol{w}^{k+1})=\frac{1}{2\gamma}\left( \norm{\boldsymbol{w}-\boldsymbol{w}^{k+1}}_{\boldsymbol{\Sigma}_{k}}^{2}-\norm{\boldsymbol{w}-\boldsymbol{w}^{k}}_{\boldsymbol{\Sigma}_{k}}^{2}\right)+(1-\frac{\gamma}{2})\norm{\tilde{\boldsymbol{w}}{}^{k}-\boldsymbol{w}^{k}}_{\boldsymbol{\Sigma}_{k}}^{2}.\label{eq46}
\end{aligned}
\end{equation}
By replacing the right-hand side of (\ref{eq42}) with (\ref{eq46}),  the assertion of this lemma is proved.
\end{proof}

To analyze the convergence rate, we require an alternative characterization of the solution set for the variational inequality (\ref{eq23}). It can be described as the following theorem.
\begin{lemma}
The solution of the variational inequality (\ref{eq23}) is convex and  it can be characterized as 
\begin{equation}
\Omega^{*}=\underset{\boldsymbol{w}\in \Omega}{\bigcap}\{\tilde{\boldsymbol{w}}\in \Omega: \vartheta(\boldsymbol{w})-\vartheta(\tilde{\boldsymbol{w}})+(\boldsymbol{w}-\tilde{\boldsymbol{w}})^{T}\boldsymbol{\Gamma}(\boldsymbol{w})\ge0\}.\label{eq47}
\end{equation}\label{lemma4}
\end{lemma}\vspace{-0.4cm}
\noindent The proof can be found in \cite{H2} (Theorem 2.1). For (\ref{eq47}), in this paper, we have 
\begin{equation}
 \boldsymbol{w}=\left( \begin{array}{c} \boldsymbol{x} \\
      \boldsymbol{\lambda} \end{array} \right), \ \vartheta(\boldsymbol{w}) = f(\boldsymbol{x}) \mbox{ and }
 \boldsymbol{\Gamma}(\boldsymbol{w})=\left( \begin{array}{c} \mathcal{D} \Phi (\boldsymbol{x})^{T}\boldsymbol{\lambda} \\
        -\Phi (\boldsymbol{x}) \end{array} \right).\label{eq48 }
\end{equation}
By Lemma \ref{lemma2}, for any $ \boldsymbol{w}, \tilde{\boldsymbol{w}}\in \Omega$, the monotone operator can be written as 
\begin{equation}
(\boldsymbol{w}-\tilde{\boldsymbol{w}})^{T}\boldsymbol{\Gamma}(\boldsymbol{w})\ge(\boldsymbol{w}-\tilde{\boldsymbol{w}})^{T}\boldsymbol{\Gamma}(\tilde{\boldsymbol{w}}).\label{eq49}
\end{equation}
Setting $\tilde{\boldsymbol{w}}=\boldsymbol{w}^{*}$ in (\ref{eq49}), we can rewrite variational inequality (\ref{eq68}) in the following equivalent form:
\begin{equation}
\boldsymbol{w}^{*}\in \Omega, \quad f(\boldsymbol{x})- f(\boldsymbol{x}^{*})+(\boldsymbol{w}-\boldsymbol{w}^{*})^{T}\boldsymbol{\Gamma}(\boldsymbol{w})\ge0, \quad\forall \ \boldsymbol{w}\in \Omega. \label{eq50}
\end{equation}
According to the equivalent characterization (\ref{eq50}), for given $\epsilon >0$, $\tilde{\boldsymbol{w}}\in \Omega$ is considered an approximate solution of variational inequality (\ref{eq23}), if it satisfies 
\begin{equation}
 f(\boldsymbol{x})-f(\tilde{\boldsymbol{x}})+(\boldsymbol{w}-\tilde{\boldsymbol{w}})^{T}\boldsymbol{\Gamma}(\boldsymbol{w})\ge-\epsilon, \quad\forall \ \boldsymbol{w}\in \mathsf{N}_{r}(\tilde{\boldsymbol{w}}),\label{eq51 }
\end{equation}
where $\mathsf{N}_{r}(\tilde{\boldsymbol{w}})=\{\boldsymbol{w}\mid \Vert\boldsymbol{w}-\tilde{\boldsymbol{w}}\Vert\leq 1\}$ is the unit neighborhood of $\tilde{\boldsymbol{w}}$. For the given $\epsilon>0$, after  $t$ iterations, it can offer a $\tilde{\boldsymbol{w}}\in \Omega$, such that  
\begin{equation}
\sup_{\boldsymbol{w}\in\mathsf{N}_{r}(\tilde{\boldsymbol{w}})} \{f(\tilde{\boldsymbol{x}})-f(\boldsymbol{x})+(\tilde{\boldsymbol{w}}-\boldsymbol{w})^{T}\boldsymbol{\Gamma}(\boldsymbol{w})\}\leq\epsilon. \label{eq52}
\end{equation}
Referring to the monotone operator $\boldsymbol{\Gamma}(\boldsymbol{w})$, formula (\ref{eq41}) can be rewritten as 
\begin{equation}
\begin{aligned}
f(\boldsymbol{x})-f(\tilde{\boldsymbol{x}}{}^{k})+(\boldsymbol{w}-\tilde{\boldsymbol{w}}{}^{k})\boldsymbol{\Gamma}(\boldsymbol{w})+\frac{1}{2\gamma}\left(\Vert \boldsymbol{w}-\boldsymbol{w}^{k} \Vert_{\boldsymbol{\Sigma}_{k}}^{2}-\Vert \boldsymbol{w}-\boldsymbol{w}^{k+1} \Vert_{\boldsymbol{\Sigma}_{k}}^{2}\right)\ge0, \quad \forall \ \boldsymbol{w}\in \Omega. \end{aligned}\label{eq53}
\end{equation}

\begin{lemma}
Given a closed convex set $\mathfrak{X}=\{\boldsymbol{x}\mid\norm{\boldsymbol{x}}<\infty,\boldsymbol{x}\in \mathcal{X}\}$, function $\mathcal{D} \Phi (\boldsymbol{x})$ is bounded on the set $\mathfrak{X}$ if  the following inequality is satisfied.
\begin{equation} \left\{ \begin{array}{ll}
\Vert \mathcal{D} \Phi (\boldsymbol{x})\Vert^{2}\leq L\norm{\boldsymbol{x}}^{2}\leq L C_{1}, &\ \forall \ \boldsymbol{x}\in \mathfrak{X},\\
\Vert \mathcal{D} \Phi (\boldsymbol{x}')-\mathcal{D} \Phi (\boldsymbol{x}'')\Vert^{2}\leq L\norm{\boldsymbol{x}'-\boldsymbol{x}''}^{2}\leq LC_{2}, &\ \forall \ \boldsymbol{x}',\boldsymbol{x}''\in \mathfrak{X},
\end{array}\right.\label{eq54}\end{equation}
where the constant $L$ belongs to the interval $(0, \infty)$, $C_{1}=\max_{\boldsymbol{x}} \norm{\boldsymbol{x}}^{2}$ and $C_{2}=\max_{\boldsymbol{x}'-\boldsymbol{x}''}\norm{\boldsymbol{x}'-\boldsymbol{x}''}^{2}$. If the sequence $\{\tilde{\boldsymbol{x}}{}^{k}\mid k=0,1,\cdots,t\}$ generated by a given method belongs to the set $\mathfrak{X}$, for $k=1,\cdots,t$, there exist a set of regularization factors $\{r_{k},s_{k}\}$ that satisfies $\boldsymbol{D}_{k}=\boldsymbol{\Sigma}_{k-1}-\boldsymbol{\Sigma}_{k}\succeq0$.\label{lemma55}
\end{lemma}
\begin{proof}
The difference matrix $\boldsymbol{D}_{k}$ is written as 
\begin{equation}\boldsymbol{D}_{k}=\boldsymbol{\Sigma}_{k-1}-\boldsymbol{\Sigma}_{k}=\left(
\begin{array}{cc}
(r_{k-1}-r_{k})\boldsymbol{I}_{n} & (\mathcal{D} \Phi (\tilde{\boldsymbol{x}}^{k})-\mathcal{D} \Phi (\tilde{\boldsymbol{x}}^{k-1}))^{T} \\
\mathcal{D} \Phi (\tilde{\boldsymbol{x}}^{k})-\mathcal{D} \Phi (\tilde{\boldsymbol{x}}^{k-1})&(s_{k-1}-s_{k})\boldsymbol{I}_{m} 
\end{array}
\right).\label{eq55 }
\end{equation}
If $\boldsymbol{D}_{k}$ is positive definite matrix, the following condition should be satisfied.
\begin{equation} \left\{ \begin{array}{lrl}
0<r_{k}\leq r_{k-1}, 0<s_{k}\leq s_{k-1},\\
r_{k}s_{k}\leq r_{k-1}s_{k-1}< L\max_{\boldsymbol{x}} \norm{\boldsymbol{x}}^{2},\\
(r_{k-1}-r_{k})(s_{k-1}-s_{k})> \Vert \mathcal{D} \Phi (\tilde{\boldsymbol{x}}^{k})-\mathcal{D} \Phi (\tilde{\boldsymbol{x}}^{k-1})\Vert^{2}.
\end{array}\right.\label{eq56}
\end{equation}
It is obvious that the equal sign is taken when %$\Phi (\boldsymbol{x})$ is a linear function, meaning that
 $\mathcal{D} \Phi (\tilde{\boldsymbol{x}}^{k})=\mathcal{D} \Phi (\tilde{\boldsymbol{x}}^{k-1})$ holds.
In other cases, we need to design a monotonically decreasing sequence that fulfills condition (\ref{eq56}). Since the values of $\{r_{k},s_{k}\}$ can be chosen arbitrarily from the interval $(0,\infty)$, such a sequence exists. For example, we take 
\begin{equation}
r_{t}=s_{t}=LC_{1}+\sigma,\label{eq57 }
\end{equation}
where $\sigma>0$ and let $r_{t-1}=s_{t-1}$. Referring to (\ref{eq56}), we have 
\begin{equation}
(r_{t-1}-r_{t})^{2}= LC_{2}+\sigma, \quad r_{t-1}=r_{t}+\sqrt{LC_{2}+\sigma}.\label{eq58 }
\end{equation}
Then the sequence can be expressed by
\begin{equation}
r_{k}=(t-k)\sqrt{LC_{2}+\sigma}+LC_{1}+\sigma, k=0,1,\cdots,t.\label{eq59 }
\end{equation}
Thus, this lemma is proved.
\end{proof}

\begin{theorem}
Let problem (\ref{eq36}) be solved by the relaxed PPA, and the generated sequence $\{\tilde{\boldsymbol{x}}^{k}\}$ follows condition (\ref{eq54}). Let $\tilde{\boldsymbol{x}}_{t}$ and $\tilde{\boldsymbol{w}}_{t}$ be defined as follows. 
\begin{equation}
\tilde{\boldsymbol{x}}_{t}=\frac{1}{1+t}\sum_{k=0}^{t}\tilde{\boldsymbol{x}}^{k},\quad \tilde{\boldsymbol{w}}_{t}=\frac{1}{1+t}\sum_{k=0}^{t}\tilde{\boldsymbol{w}}^{k}.\label{eq60 }
\end{equation}
If the regularization factors satisfy Lemma \ref{lemma55}, for any integer number $t>0$, we have $\boldsymbol{D}_{k}=\boldsymbol{\Sigma}_{k-1}-\boldsymbol{\Sigma}_{k}\succ0$ and
\begin{equation}
\tilde{\boldsymbol{w}}_{t}\in \Omega,\quad f(\tilde{\boldsymbol{x}}_{t})-f(\boldsymbol{x})+(\tilde{\boldsymbol{w}}_{t}-\boldsymbol{w})^{T}\boldsymbol{\Gamma}(\boldsymbol{w})\leq \frac{1}{2\gamma(1+t)}\Vert\boldsymbol{w}-\boldsymbol{w}^{0}\Vert_{\boldsymbol{\Sigma}_{0}}^{2}, \quad\forall \ \boldsymbol{w}\in \Omega. \label{eq61}
\end{equation}
\end{theorem}
\begin{proof}
For all $k=0, 1,\cdots,t$, $\tilde{\boldsymbol{w}}^{k}$ belongs to the convex set $\Omega$. Then the linear combination $\tilde{\boldsymbol{w}}_{t}\in \Omega$ holds.  Summing inequality (\ref{eq53}) over $k=0,1,\cdots,t$, for any $ \boldsymbol{w}\in \Omega$, we get 
\begin{equation}
(1+t)f(\boldsymbol{x})-\sum_{k=0}^{t}f(\tilde{\boldsymbol{x}}{}^{k})+\left((1+t)\boldsymbol{w}-\sum_{k=0}^{t}\tilde{\boldsymbol{w}}{}^{k}\right)\boldsymbol{\Gamma}(\boldsymbol{w})+\frac{1}{2\gamma}\sum_{k=0}^{t}\left(\Vert \boldsymbol{w}-\boldsymbol{w}^{k} \Vert_{\boldsymbol{\Sigma}_{k}}^{2}-\Vert \boldsymbol{w}-\boldsymbol{w}^{k+1} \Vert_{\boldsymbol{\Sigma}_{k}}^{2}\right)\ge 0.
\end{equation}
The inequality above divided by $1+t$ equals 
\begin{equation}
\begin{aligned}
f(\boldsymbol{x})-\frac{1}{1+t}\sum_{k=0}^{t}f(\tilde{\boldsymbol{x}}{}^{k})+\left(\boldsymbol{w}-\frac{1}{1+t}\sum_{k=0}^{t}\tilde{\boldsymbol{w}}{}^{k}\right)\boldsymbol{\Gamma}(\boldsymbol{w})+\frac{1}{2\gamma (1+t)}\sum_{k=0}^{t}\left(\Vert \boldsymbol{w}-\boldsymbol{w}^{k} \Vert_{\boldsymbol{\Sigma}_{k}}^{2}-\Vert \boldsymbol{w}-\boldsymbol{w}^{k+1} \Vert_{\boldsymbol{\Sigma}_{k}}^{2}\right)\ge 0.
\end{aligned}%\tag{61b}
\end{equation}
We know that $f(\boldsymbol{x})$ is convex in $\mathcal{X}$ and $\tilde{\boldsymbol{x}}_{t}=\frac{1}{1+t}\sum_{k=0}^{t}\tilde{\boldsymbol{x}}^{k}\in \mathcal{X}$. So  $f(\tilde{\boldsymbol{x}}_{t})\leq\frac{1}{1+t}\sum_{k=0}^{t}f(\tilde{\boldsymbol{x}}^{k})$ holds.
\begin{equation}
\begin{aligned}
 f(\boldsymbol{x})-f(\tilde{\boldsymbol{x}}_{t})+\left(\boldsymbol{w}-\tilde{\boldsymbol{w}}_{t}\right)\boldsymbol{\Gamma}(\boldsymbol{w})+\frac{1}{2\gamma (1+t)}\sum_{k=0}^{t}\left(\Vert \boldsymbol{w}-\boldsymbol{w}^{k} \Vert_{\boldsymbol{\Sigma}_{k}}^{2}-\Vert \boldsymbol{w}-\boldsymbol{w}^{k+1} \Vert_{\boldsymbol{\Sigma}_{k}}^{2}\right)&\ge 0, \\ f(\boldsymbol{x})-f(\tilde{\boldsymbol{x}}_{t})+\left(\boldsymbol{w}-\tilde{\boldsymbol{w}}_{t}\right)\boldsymbol{\Gamma}(\boldsymbol{w})+\frac{1}{2\gamma (1+t)}\left(\Vert \boldsymbol{w}-\boldsymbol{w}^{0} \Vert_{\boldsymbol{\Sigma}_{0}}^{2}-\sum_{k=1}^{t}(\boldsymbol{w}-\boldsymbol{w}^{k})^{T}\boldsymbol{D}_{k}(\boldsymbol{w}-\boldsymbol{w}^{k})\right)&\ge 0, \\
\end{aligned}%\tag{61c}%\label{eq62}
\end{equation}
where $\boldsymbol{D}_{k}=\boldsymbol{\Sigma}_{k-1}-\boldsymbol{\Sigma}_{k}$. Since $\boldsymbol{D}_{k}\succ0$, $\Vert\boldsymbol{w}-\boldsymbol{w}^{k}\Vert_{\boldsymbol{D}_{k}}\ge0$ holds. Then we have 
\begin{equation}
\begin{aligned}
f(\boldsymbol{x})-f(\tilde{\boldsymbol{x}}_{t})+\left(\boldsymbol{w}-\tilde{\boldsymbol{w}}_{t}\right)\boldsymbol{\Gamma}(\boldsymbol{w})+\frac{1}{2\gamma (1+t)}\Vert \boldsymbol{w}-\boldsymbol{w}^{0} \Vert_{\boldsymbol{\Sigma}_{0}}^{2}&\ge \frac{1}{2\gamma (1+t)}\sum_{k=1}^{t}\Vert\boldsymbol{w}-\boldsymbol{w}^{k}\Vert_{\boldsymbol{D}_{k}}, \\
f(\boldsymbol{x})-f(\tilde{\boldsymbol{x}}_{t})+\left(\boldsymbol{w}-\tilde{\boldsymbol{w}}_{t}\right)\boldsymbol{\Gamma}(\boldsymbol{w})+\frac{1}{2\gamma (1+t)}\Vert \boldsymbol{w}-\boldsymbol{w}^{0} \Vert_{\boldsymbol{\Sigma}_{0}}^{2}&\ge 0.\\
\end{aligned}\label{eq63}
\end{equation}
Thus, this theorem holds.
\end{proof}

\begin{remark}
When all constraints are linear, indicating that  $\mathcal{D} \Phi (\tilde{\boldsymbol{x}}^{k})=\mathcal{D} \Phi (\tilde{\boldsymbol{x}}^{k-1})$ holds for $k=1,\cdots,t$, the difference matrix $\boldsymbol{D}_{k}$ becomes 
\begin{equation}\boldsymbol{D}_{k}=\left(
\begin{array}{cc}
(r_{k-1}-r_{k})\boldsymbol{I}_{n} & 0 \\
0&(s_{k-1}-s_{k})\boldsymbol{I}_{m} 
\end{array}
\right).\label{eq64 }
\end{equation}
If $r_{k}$ and $s_{k}$ are taken as two constants, we get $\boldsymbol{D}_{k}=0$. Therefore, for linear constrained convex problems, the relaxed PPA can produce an approximate solution with an accuracy of $O(1/t)$ after $t$ iterations when we take $r_{k}=r$,  $s_{k}=s$ and $rs>\Vert\mathcal{D} \Phi (\tilde{\boldsymbol{x}}^{k})\Vert^{2}$.
\end{remark}

\begin{algorithm}[t]\label{algorithm1}
\caption{Relaxed PPA method}
\LinesNumbered
\KwIn{ Parameters $\mu_{1}>1,\mu_{2}>1,\gamma\in(0, 2)$; tolerant error $\tau$.}
\KwOut{The optimal solution: $\boldsymbol{x}^{*}$.}
Initialize $\boldsymbol{x}^{0},\boldsymbol{\lambda}^{0},k=0$\;
\While{$error\ge\tau$}{
\emph{\% The customized PPA step:}\\
$r_{k}=\frac{1}{\mu_{1}}\Vert\mathcal{D} \Phi (\boldsymbol{x}^{k})\Vert$\;
$\tilde{\boldsymbol{x}}^{k}=\arg\min\{\mathcal{L}''(\boldsymbol{x},\boldsymbol{\lambda}^{k})+\frac{r_{k}}{2}\norm{\boldsymbol{x}-\boldsymbol{x}^{k}}^{2}\mid \boldsymbol{x}\in \mathcal{X}\}$\;
$s_{k}=\frac{\mu_{2}}{r_{k}}\Vert\mathcal{D} \Phi (\tilde{\boldsymbol{x}}^{k})\Vert^{2}$\;
$\tilde{\boldsymbol{\lambda}}^{k}=P_{\mathcal{Z}}\left(\boldsymbol{\lambda}^{k}+\frac{1}{s_{k}}[  \Phi (\tilde{\boldsymbol{x}}^{k}) +\mathcal{D} \Phi (\tilde{\boldsymbol{x}}^{k})(\tilde{\boldsymbol{x}}^{k}-\boldsymbol{x}^{k})]\right)$\;
\emph{\% The relaxed step:}\\
$\boldsymbol{w}^{k+1}=\boldsymbol{w}^{k}-\gamma(\boldsymbol{w}^{k}-\tilde{\boldsymbol{w}}^{k})$\;
$error=\mbox{abs}[ f(\boldsymbol{x}^{k})- f(\boldsymbol{x}^{k+1})]$\;
$k=k+1$\;
} 
\Return { $\boldsymbol{x}^{*}=\boldsymbol{x}^{k}$}\;
 \end{algorithm}

\subsection{Regularization Parameter Selection}
Based on the convergence analysis, it is imperative that the proximal matrix be symmetrical and positive definite. However, achieving the convergence rate of $O(1/t)$ is not easy due to the parameter selection of ${r_{k},s_{k}}$ relying on the unknown constant $L$. In general, we adjust the values of $\{r_{k},s_{k}\}$ for the $k$-th iteration using the following formulae.
 \begin{equation}
r_{k}=\frac{1}{\mu_{1}}\Vert\mathcal{D} \Phi (\boldsymbol{x}^{k})\Vert,\label{eq65}
\end{equation}
\begin{equation}
s_{k}=\frac{\mu_{2}}{r_{k}}\Vert\mathcal{D} \Phi (\tilde{\boldsymbol{x}}^{k})\Vert^{2},\label{eq66}
\end{equation}
where $\mu_{1}>1$ and $\mu_{2}>1$ are two positive factor parameters. %We generally take $\mu_{1}>1$ to avoid a large value of the proximal term. 
Based on the truth that $\rho(\mathcal{D} \Phi (\tilde{\boldsymbol{x}}{}^{k})\mathcal{D} \Phi (\tilde{\boldsymbol{x}}{}^{k})^{T})\leq\Vert\mathcal{D} \Phi (\tilde{\boldsymbol{x}}^{k})\Vert^{2}$, we multiply equations (\ref{eq65}) and (\ref{eq66}) to obtain the result. 
\begin{equation}r_{k}s_{k}=\mu_{2}\Vert\mathcal{D} \Phi (\tilde{\boldsymbol{x}}^{k})\Vert^{2}>\Vert\mathcal{D} \Phi (\tilde{\boldsymbol{x}}^{k})\Vert^{2}.\label{eq67}\end{equation} 
The above inequality ensures that the symmetric proximal matrix $\boldsymbol{\Sigma}_{k}$ is positive definite. The pseudo-code of the relaxed PPA can be found in Algorithm \ref{algorithm1}.

\section{Prediction Correction Method}\label{section4}
As described in Section \ref{section2}, Lagrangian function (\ref{eq19}) results in a variational inequality:
\begin{equation}
\boldsymbol{w}^{*}\in \Omega, \quad f(\boldsymbol{x})- f(\boldsymbol{x}^{*})+(\boldsymbol{w}-\boldsymbol{w}^{*})^{T}\boldsymbol{\Gamma}(\boldsymbol{w}^{*})\ge0, \quad\forall \ \boldsymbol{w}\in \Omega. \label{eq68}
\end{equation}
The PC method consists of two steps: the prediction step and the correction step. In the prediction step, we take the output of the primal-dual method, $ \tilde{\boldsymbol{w}}^{k}=( \tilde{\boldsymbol{x}}^{k}, \tilde{\boldsymbol{\lambda}}^{k})$, as the predictive variables by solving (\ref{eq25}). From (\ref{eq28}), the variational inequality can be written as 
\begin{equation}
 \tilde{\boldsymbol{w}}^{k}\in \Omega, \quad f(\boldsymbol{x})- f(\tilde{\boldsymbol{x}}^{k})+(\boldsymbol{w}-\tilde{\boldsymbol{w}}^{k})^{T}\boldsymbol{\Gamma}(\tilde{\boldsymbol{w}}^{k})\ge(\boldsymbol{w}-\tilde{\boldsymbol{w}}^{k})^{T}\boldsymbol{Q}_{k}(\boldsymbol{w}^{k}-\tilde{\boldsymbol{w}}^{k}), \quad\forall \ \boldsymbol{w}\in \Omega, \label{eq69}
\end{equation}
where $\boldsymbol{Q}_{k}$ is called as the predictive matrix. When condition (\ref{eq67}) holds, the symmetrical matrix
\begin{equation} \boldsymbol{Q}_{k}^{T}+\boldsymbol{Q}_{k}=\left(
\begin{array}{cc}
2r_{k}\boldsymbol{I}_{n} & -\mathcal{D} \Phi (\tilde{\boldsymbol{x}}^{k})^{T} \\
-\mathcal{D} \Phi (\tilde{\boldsymbol{x}}^{k})&2s_{k}\boldsymbol{I}_{m} 
\end{array}
\right)\succ0 \mbox{ is positive definite.}\nonumber
\end{equation}
In the correction step, we use a corrective matrix $\boldsymbol{M}$ to correct the predictive variables by the following equation.
\begin{equation}
\boldsymbol{w}^{k+1}=\boldsymbol{w}^{k}-\beta\boldsymbol{M}_{k}(\boldsymbol{w}^{k}-\tilde{\boldsymbol{w}}^{k}),\label{eq70}
\end{equation}
where $\beta>0$ is the correction step size and the corrective matrix $\boldsymbol{M}$ is not unique \cite{E4}.

\subsection{Convergence Analysis for the PC method}
\begin{theorem}
Let $\{\boldsymbol{w}^{k},\tilde{\boldsymbol{w}}^{k},\boldsymbol{w}^{k+1}\}$ be generated by the PC method. For the predictive matrix $\boldsymbol{Q}_{k}$, if there is a corrective matrix $\boldsymbol{M}$ that satisfies the following convergence condition.
\begin{equation}
\boldsymbol{\Sigma}_{k}=\boldsymbol{Q}_{k}\boldsymbol{M}_{k}^{-1}\succ0\mbox{ and }
\boldsymbol{G}_{k}=\boldsymbol{Q}_{k}^{T}+\boldsymbol{Q}_{k}-\beta\boldsymbol{M}_{k}^{T}\boldsymbol{\Sigma}_{k}\boldsymbol{M}_{k}\succ0,\label{eq71}
\end{equation}
 then we have 
\begin{equation}
 \beta[ f(\boldsymbol{x})- f(\tilde{\boldsymbol{x}}^{k})+(\boldsymbol{w}-\tilde{\boldsymbol{w}}^{k})^{T}\boldsymbol{\Gamma}(\tilde{\boldsymbol{w}}^{k})]\ge\frac{1}{2}\left(\norm{\boldsymbol{w}-\boldsymbol{w}^{k+1}}_{\boldsymbol{\Sigma}_{k}}^{2}-\norm{\boldsymbol{w}-\boldsymbol{w}^{k}}_{\boldsymbol{\Sigma}_{k}}^{2}\right) +\frac{\beta}{2}\norm{\tilde{\boldsymbol{w}}{}^{k}-\boldsymbol{w}^{k}}_{\boldsymbol{G}_{k}}^{2},\quad \forall \ \boldsymbol{w}\in \Omega. \label{eq72}
\end{equation}
\end{theorem}
\begin{proof}
By (\ref{eq70}), we have 
\begin{equation}
(\boldsymbol{w}-\tilde{\boldsymbol{w}}^{k})^{T}\boldsymbol{Q}_{k}(\boldsymbol{w}^{k}-\tilde{\boldsymbol{w}}^{k})=(\boldsymbol{w}-\tilde{\boldsymbol{w}}^{k})^{T}\boldsymbol{\Sigma}_{k}\boldsymbol{M}_{k}(\boldsymbol{w}^{k}-\tilde{\boldsymbol{w}}^{k})=(\boldsymbol{w}-\tilde{\boldsymbol{w}}^{k})^{T}\frac{1}{\beta}\boldsymbol{\Sigma}_{k}(\boldsymbol{w}^{k}-\boldsymbol{w}^{k+1}).\label{eq73}
\end{equation}
Applying (\ref{eq43}) to the equation (\ref{eq73}), we obtain
\begin{equation}
\begin{aligned}
(\boldsymbol{w}-\tilde{\boldsymbol{w}}^{k})^{T}\boldsymbol{\Sigma}_{k}(\boldsymbol{w}^{k}-\boldsymbol{w}^{k+1})=&\frac{1}{2}\left( \norm{\boldsymbol{w}-\boldsymbol{w}^{k+1}}_{\boldsymbol{\Sigma}_{k}}^{2}-\norm{\boldsymbol{w}-\boldsymbol{w}^{k}}_{\boldsymbol{\Sigma}_{k}}^{2}\right)\\
&+\frac{1}{2}\left( \norm{\tilde{\boldsymbol{w}}{}^{k}-\boldsymbol{w}^{k}}_{\boldsymbol{\Sigma}_{k}}^{2}-\norm{\tilde{\boldsymbol{w}}{}^{k}-\boldsymbol{w}^{k+1}}_{\boldsymbol{\Sigma}_{k}}^{2}\right).\label{eq74 }
\end{aligned}
\end{equation}
On the other hand, the following equation holds.
\begin{eqnarray}
\lefteqn{\norm{\tilde{\boldsymbol{w}}{}^{k}-\boldsymbol{w}^{k}}_{\boldsymbol{\Sigma}_{k}}^{2}-\norm{\tilde{\boldsymbol{w}}{}^{k}-\boldsymbol{w}^{k+1}}_{\boldsymbol{\Sigma}_{k}}^{2}}\nonumber\\
&=&\norm{\tilde{\boldsymbol{w}}{}^{k}-\boldsymbol{w}^{k}}_{\boldsymbol{\Sigma}_{k}}^{2}-\norm{(\tilde{\boldsymbol{w}}{}^{k}-\boldsymbol{w}^{k})-(\boldsymbol{w}^{k+1}-\boldsymbol{w}^{k})}_{\boldsymbol{\Sigma}_{k}}^{2}\nonumber\\
&=&\norm{\tilde{\boldsymbol{w}}{}^{k}-\boldsymbol{w}^{k}}_{\boldsymbol{\Sigma}_{k}}^{2}-\norm{(\tilde{\boldsymbol{w}}{}^{k}-\boldsymbol{w}^{k})-\beta\boldsymbol{M}_{k}(\tilde{\boldsymbol{w}}{}^{k}-\boldsymbol{w}^{k})}_{\boldsymbol{\Sigma}_{k}}^{2}\\
&=&2\beta(\tilde{\boldsymbol{w}}{}^{k}-\boldsymbol{w}^{k})^{T}\boldsymbol{\Sigma}_{k}\boldsymbol{M}_{k}(\tilde{\boldsymbol{w}}{}^{k}-\boldsymbol{w}^{k})-\beta^{2}(\tilde{\boldsymbol{w}}{}^{k}-\boldsymbol{w}^{k})^{T}\boldsymbol{M}_{k}^{T}\boldsymbol{\Sigma}_{k}\boldsymbol{M}_{k}(\tilde{\boldsymbol{w}}{}^{k}-\boldsymbol{w}^{k})\nonumber\\
&=&\beta(\tilde{\boldsymbol{w}}{}^{k}-\boldsymbol{w}^{k})^{T}(\boldsymbol{Q}_{k}^{T}+\boldsymbol{Q}_{k}-\beta\boldsymbol{M}_{k}^{T}\boldsymbol{\Sigma}_{k}\boldsymbol{M}_{k})(\tilde{\boldsymbol{w}}{}^{k}-\boldsymbol{w}^{k})\nonumber\\
&=&\beta\norm{\tilde{\boldsymbol{w}}{}^{k}-\boldsymbol{w}^{k}}_{\boldsymbol{G}_{k}}^{2}.\nonumber
\label{eq75 }
\end{eqnarray}
Combining with (\ref{eq73}), we have
\begin{equation}
\begin{aligned}
(\boldsymbol{w}-\tilde{\boldsymbol{w}}^{k})^{T}\boldsymbol{Q}_{k}(\boldsymbol{w}^{k}-\tilde{\boldsymbol{w}}^{k})=\frac{1}{2\beta}\left( \norm{\boldsymbol{w}-\boldsymbol{w}^{k+1}}_{\boldsymbol{\Sigma}_{k}}^{2}-\norm{\boldsymbol{w}-\boldsymbol{w}^{k}}_{\boldsymbol{\Sigma}_{k}}^{2}\right)+\frac{1}{2}\norm{\tilde{\boldsymbol{w}}{}^{k}-\boldsymbol{w}^{k}}_{\boldsymbol{G}_{k}}^{2},\label{eq76}
\end{aligned}
\end{equation}
By replacing the right-hand side of (\ref{eq69}) with (\ref{eq76}), the assertion of this theorem is proved.
\end{proof}

\begin{theorem}[Sequence contraction]
Let $\{\boldsymbol{w}^{k},\tilde{\boldsymbol{w}}^{k},\boldsymbol{w}^{k+1}\}$ be generated by the PC method. For the predictive matrix $\boldsymbol{Q}_{k}$, if there is a corrective matrix $\boldsymbol{M}_{k}$ that satisfies the convergence condition (\ref{eq71}), then we have 
\begin{equation}
\Vert\boldsymbol{w}^{*}-\boldsymbol{w}^{k+1}\Vert_{\boldsymbol{\Sigma}_{k}}^{2}\leq\Vert\boldsymbol{w}^{*}-\boldsymbol{w}^{k}\Vert_{\boldsymbol{\Sigma}_{k}}^{2}-\beta\Vert\boldsymbol{w}^{k}-\tilde{\boldsymbol{w}}^{k}\Vert_{\boldsymbol{G}_{k}}^{2}, \quad\forall \ \boldsymbol{w}^{*}\in \Omega^{*}, \label{eq77}
\end{equation}
where $\Omega^{*}$ is the set of optimal solutions.
\end{theorem}
\begin{proof} Setting $\boldsymbol{w}=\boldsymbol{w}^{*}$ in (\ref{eq72}), we get 
\begin{equation}
\norm{\boldsymbol{w}^{*}-\boldsymbol{w}^{k}}_{\boldsymbol{\Sigma}_{k}}^{2}-\norm{\boldsymbol{w}^{*}-\boldsymbol{w}^{k+1}}_{\boldsymbol{\Sigma}_{k}}^{2}-\beta\Vert\boldsymbol{w}^{k}-\tilde{\boldsymbol{w}}^{k}\Vert_{\boldsymbol{G}_{k}}^{2}\ge 2\beta[ f(\tilde{\boldsymbol{x}}^{k})- f(\boldsymbol{x}^{*})+(\tilde{\boldsymbol{w}}^{k}-\boldsymbol{w}^{*})^{T}\boldsymbol{\Gamma}(\tilde{\boldsymbol{w}}^{k})].\label{eq78 }
\end{equation}
By Lemme 2, the monotone operator satisfies $(\tilde{\boldsymbol{w}}^{k}-\boldsymbol{w}^{*})^{T}\boldsymbol{\Gamma}(\tilde{\boldsymbol{w}}^{k})\ge (\tilde{\boldsymbol{w}}^{k}-\boldsymbol{w}^{*})^{T}\boldsymbol{\Gamma}(\boldsymbol{w}^{*}).$
\begin{equation}
\Vert\boldsymbol{w}^{*}-\boldsymbol{w}^{k}\Vert_{\boldsymbol{\Sigma}_{k}}^{2}-\Vert\boldsymbol{w}^{*}-\boldsymbol{w}^{k+1}\Vert_{\boldsymbol{\Sigma}_{k}}^{2}-\beta\Vert\boldsymbol{w}^{k}-\tilde{\boldsymbol{w}}^{k}\Vert_{\boldsymbol{G}_{k}}^{2}\ge 2\beta[ f(\tilde{\boldsymbol{x}}^{k})- f(\boldsymbol{x}^{*})+(\tilde{\boldsymbol{w}}^{k}-\boldsymbol{w}^{*})^{T}\boldsymbol{\Gamma}(\boldsymbol{w}^{*})]\ge0.\label{eq79 }
\end{equation}
Then we have 
\begin{equation}
\Vert\boldsymbol{w}^{*}-\boldsymbol{w}^{k}\Vert_{\boldsymbol{\Sigma}_{k}}^{2}-\Vert\boldsymbol{w}^{*}-\boldsymbol{w}^{k+1}\Vert_{\boldsymbol{\Sigma}_{k}}^{2}-\beta\Vert\boldsymbol{w}^{k}-\tilde{\boldsymbol{w}}^{k}\Vert_{\boldsymbol{G}_{k}}^{2}\ge0.\label{eq80 }
\end{equation}
Thus, this theorem holds.
\end{proof}
Lemma \ref{lemma4} is also the base for the convergence rate proof.  Referring to (\ref{eq49}), we set $\tilde{\boldsymbol{w}}=\tilde{\boldsymbol{w}}^{k}$ and add it to (\ref{eq72}). Then we can obtain the following inequality.
\begin{equation}
f(\boldsymbol{x})- f(\tilde{\boldsymbol{x}}^{k})+(\boldsymbol{w}-\tilde{\boldsymbol{w}}^{k})^{T}\boldsymbol{\Gamma}(\boldsymbol{w})+\frac{1}{2\beta}\norm{\boldsymbol{w}-\boldsymbol{w}^{k}}_{\boldsymbol{\Sigma}_{k}}^{2}\ge\frac{1}{2\beta}\norm{\boldsymbol{w}-\boldsymbol{w}^{k+1}}_{\boldsymbol{\Sigma}_{k}}^{2}, \quad\forall \ \boldsymbol{w}\in \Omega. \label{eq81}
\end{equation}
Note that the above assertion holds only for  $\boldsymbol{G}_{k}\succ 0$.

\begin{theorem}
Let problem (\ref{eq68}) be solved by the PC method, and the generated sequence $\{\tilde{\boldsymbol{x}}^{k}\}$ follows condition (\ref{eq54}). Let $\tilde{\boldsymbol{x}}_{t}$ and $\tilde{\boldsymbol{w}}_{t}$ be defined as follows. 
\begin{equation}
\tilde{\boldsymbol{x}}_{t}=\frac{1}{1+t}\sum_{k=0}^{t}\tilde{\boldsymbol{x}}^{k},\quad \tilde{\boldsymbol{w}}_{t}=\frac{1}{1+t}\sum_{k=0}^{t}\tilde{\boldsymbol{w}}^{k}.\label{eq82 }
\end{equation}
If the regularization factors satisfy Lemma \ref{lemma55} and the condition(\ref{eq71}) holds, for any integer number $t>0$, we have 
\begin{equation}
\tilde{\boldsymbol{w}}_{t}\in \Omega,\quad f(\tilde{\boldsymbol{x}}_{t})-f(\boldsymbol{x})+(\tilde{\boldsymbol{w}}_{t}-\boldsymbol{w})^{T}\boldsymbol{\Gamma}(\boldsymbol{w})\leq \frac{1}{2\beta(1+t)}\Vert\boldsymbol{w}-\boldsymbol{w}^{0}\Vert_{\boldsymbol{\Sigma}_{0}}^{2}, \quad\forall \ \boldsymbol{w}\in \Omega. \label{eq83}
\end{equation}
\end{theorem}
\begin{proof}
For all $k=0, 1,\cdots,t$, $\tilde{\boldsymbol{w}}^{k}$ belongs to the convex set $\Omega$. Then the linear combination $\tilde{\boldsymbol{w}}_{t}\in \Omega$ holds.  Summing inequality (\ref{eq81}) over $k=0,1,\cdots,t$, for any $ \boldsymbol{w}\in \Omega$, we get 
\begin{equation}
%\begin{align}
(1+t)f(\boldsymbol{x})- \sum_{k=0}^{t}f(\tilde{\boldsymbol{x}}^{k})+\left((1+t)\boldsymbol{w}-\sum_{k=0}^{t}\tilde{\boldsymbol{w}}^{k}\right)^{T}\boldsymbol{\Gamma}(\boldsymbol{w})+\frac{1}{2\beta}\sum_{k=0}^{t}\norm{\boldsymbol{w}-\boldsymbol{w}^{k}}_{\boldsymbol{\Sigma}_{k}}^{2}\ge\frac{1}{2\beta}\sum_{k=0}^{t}\norm{\boldsymbol{w}-\boldsymbol{w}^{k+1}}_{\boldsymbol{\Sigma}_{k}}^{2}.
%\end{align}
\end{equation}
The inequality above divided by $1+t$ equals 
\begin{equation}\begin{aligned}
f(\boldsymbol{x})- \frac{1}{1+t}\sum_{k=0}^{t}f(\tilde{\boldsymbol{x}}^{k})+\left(\boldsymbol{w}-\tilde{\boldsymbol{w}}_{t}\right)^{T}\boldsymbol{\Gamma}(\boldsymbol{w})+\frac{1}{2\beta(1+t)}\sum_{k=0}^{t}\norm{\boldsymbol{w}-\boldsymbol{w}^{k}}_{\boldsymbol{\Sigma}_{k}}^{2}&\ge\frac{1}{2\beta(1+t)}\sum_{k=0}^{t}\norm{\boldsymbol{w}-\boldsymbol{w}^{k+1}}_{\boldsymbol{\Sigma}_{k}}^{2}.
 \end{aligned}%\tag{83b}
 \end{equation}
We know that $f(\boldsymbol{x})$ is convex in $\mathcal{X}$ and $\tilde{\boldsymbol{x}}_{t}=\frac{1}{1+t}\sum_{k=0}^{t}\tilde{\boldsymbol{x}}^{k}\in \mathcal{X}$. So  $f(\tilde{\boldsymbol{x}}_{t})\leq\frac{1}{1+t}\sum_{k=0}^{t}f(\tilde{\boldsymbol{x}}^{k})$ holds.
\begin{equation}\begin{aligned}
%f(\boldsymbol{x})- \frac{1}{1+t}\sum_{k=0}^{t}f(\tilde{\boldsymbol{x}}^{k})+\left(\boldsymbol{w}-\tilde{\boldsymbol{w}}_{t}\right)^{T}\boldsymbol{\Gamma}(\boldsymbol{w})+\frac{1}{2\beta(1+t)}\sum_{k=0}^{t}\norm{\boldsymbol{w}-\boldsymbol{w}^{k}}_{\boldsymbol{\Sigma}_{k}}^{2}&\ge\frac{1}{2\beta(1+t)}\sum_{k=0}^{t}\norm{\boldsymbol{w}-\boldsymbol{w}^{k+1}}_{\boldsymbol{\Sigma}_{k}}^{2}\\
%f(\boldsymbol{x})- f(\tilde{\boldsymbol{x}}^{k}_{t})+\left(\boldsymbol{w}-\tilde{\boldsymbol{w}}_{t}\right)^{T}\boldsymbol{\Gamma}(\boldsymbol{w})+\frac{1}{2\beta(1+t)}\left(\norm{\boldsymbol{w}-\boldsymbol{w}^{0}}_{\boldsymbol{\Sigma}_{0}}^{2}-\sum_{k=1}^{t}\norm{\boldsymbol{w}-\boldsymbol{w}^{k}}_{\boldsymbol{D}_{k}}^{2}\right)&\ge\frac{1}{2\beta(1+t)}\norm{\boldsymbol{w}-\boldsymbol{w}^{t+1}}_{\boldsymbol{\Sigma}_{t}}^{2}\\
f(\boldsymbol{x})- f(\tilde{\boldsymbol{x}}^{k}_{t})+\left(\boldsymbol{w}-\tilde{\boldsymbol{w}}_{t}\right)^{T}\boldsymbol{\Gamma}(\boldsymbol{w})+\frac{1}{2\beta(1+t)}\norm{\boldsymbol{w}-\boldsymbol{w}^{0}}_{\boldsymbol{\Sigma}_{0}}^{2}
&\ge\frac{1}{2\beta(1+t)}\sum_{k=1}^{t}\norm{\boldsymbol{w}-\boldsymbol{w}^{k}}_{\boldsymbol{D}_{k}}^{2}\\
f(\boldsymbol{x})- f(\tilde{\boldsymbol{x}}^{k}_{t})+\left(\boldsymbol{w}-\tilde{\boldsymbol{w}}_{t}\right)^{T}\boldsymbol{\Gamma}(\boldsymbol{w})+\frac{1}{2\beta(1+t)}\norm{\boldsymbol{w}-\boldsymbol{w}^{0}}_{\boldsymbol{\Sigma}_{0}}^{2}
&\ge 0.
%\frac{1}{1+t}\sum_{k=0}^{t}f(\tilde{\boldsymbol{x}}^{k})-f(\boldsymbol{x})+\left(\tilde{\boldsymbol{w}}_{t}-\boldsymbol{w}\right)^{T}\boldsymbol{\Gamma}(\boldsymbol{w})&\leq\frac{1}{2\beta(1+t)}\norm{\boldsymbol{w}-\boldsymbol{w}^{0}}_{\boldsymbol{\Sigma}_{0}}^{2}\\
 %f(\tilde{\boldsymbol{x}}_{t})-f(\boldsymbol{x})+\left(\tilde{\boldsymbol{w}}_{t}-\boldsymbol{w}\right)^{T}\boldsymbol{\Gamma}(\boldsymbol{w})&\leq \frac{1}{2\beta(1+t)}\norm{\boldsymbol{w}-\boldsymbol{w}^{0}}_{\boldsymbol{\Sigma}_{0}}^{2},
 \end{aligned}%\tag{83c}%
 \label{eq84}
\end{equation}
Thus, this theorem holds.
\end{proof}
\begin{remark}
Referring back to equation (\ref{eq52}), the conclusion (\ref{eq83}) clearly indicates that the method can produce an approximate solution with an accuracy of $O(1/t)$ after $t$ iterations. In other words, the convergence rate of the method is established at $O(1/t)$ if and only if the regularization factors satisfy Lemma \ref{lemma55}. 
\end{remark}
%\begin{theorem}\begin{equation}
%\tilde{\boldsymbol{w}}_{t}\in \Omega, f(\tilde{\boldsymbol{x}}_{t})-f(\boldsymbol{x}^{*})+(\tilde{\boldsymbol{w}}_{t}-\boldsymbol{w}^{*})^{T}\boldsymbol{\Gamma}(\boldsymbol{w}^{*})\leq \frac{1}{2\beta(1+t)}\norm{\boldsymbol{w}^{*}-\boldsymbol{w}^{0}}_{\boldsymbol{\Sigma}_{k}}^{2}, 
%\end{equation}
%\end{theorem}
%\begin{equation}
%\tilde{\boldsymbol{w}}_{t}\in \Omega, f(\tilde{\boldsymbol{x}}_{t})-f(\boldsymbol{x}^{*})-\nabla f(\boldsymbol{x}^{*})^{T}(\tilde{\boldsymbol{x}}_{t}-\boldsymbol{x}^{*})-\Phi(\boldsymbol{x}^{*})^{T}\tilde{\boldsymbol{\lambda}}_{t}\leq \frac{1}{2\beta(1+t)}\norm{\boldsymbol{w}^{*}-\boldsymbol{w}^{0}}_{\boldsymbol{\Sigma}_{k}}^{2}, 
%\end{equation}

\subsection{Corrective Matrix Selection}\label{se42}
In \cite{E4}, they provide several choices for corrective matrix $\boldsymbol{M}$. For the first choice, we take 
\begin{equation} \boldsymbol{M}_{k}=\left(
\begin{array}{cc}
\boldsymbol{I}_{n} & \frac{-1}{r_{k}}\mathcal{D} \Phi (\tilde{\boldsymbol{x}}^{k})^{T} \\0&\boldsymbol{I}_{m} 
\end{array}\right).\label{eq85 }
\end{equation}
The matrix $\boldsymbol{\Sigma}_{k}$ is positive-definite for any $r_{k}, s_{k}>0$.
\begin{equation} \boldsymbol{\Sigma}_{k}=\boldsymbol{Q}_{k}\boldsymbol{M}_{k}^{-1}=\left(
\begin{array}{cc}
r_{k}\boldsymbol{I}_{n} & 0 \\
0&s_{k}\boldsymbol{I}_{m} 
\end{array}
\right)\succ0.\label{eq86 }
\end{equation}
When the condition $(2-\beta)^{2}r_{k}s_{k}>\norm{\mathcal{D} \Phi (\tilde{\boldsymbol{x}}{}^{k})}^{2}$ holds, we have
\begin{equation} 
\boldsymbol{G}_{k}=\boldsymbol{Q}_{k}^{T}+\boldsymbol{Q}_{k}-\beta\boldsymbol{M}_{k}^{T}\boldsymbol{\Sigma}_{k}\boldsymbol{M}_{k}=\left(
\begin{array}{cc}
(2-\beta)r_{k}\boldsymbol{I}_{n} & (\beta-1)\mathcal{D} \Phi (\tilde{\boldsymbol{x}}^{k})^{T} \\ (\beta-1)\mathcal{D} \Phi (\tilde{\boldsymbol{x}}^{k})&(2-\beta)s_{k}\boldsymbol{I}_{m} -\frac{\beta}{r_{k}}\mathcal{D} \Phi (\tilde{\boldsymbol{x}}^{k})\mathcal{D} \Phi (\tilde{\boldsymbol{x}}^{k})^{T}
\end{array}
\right)\succ0.\label{eq87 }
\end{equation}
For the second choice, we take 
\begin{equation} \boldsymbol{M}_{k}=\left(
\begin{array}{cc}
\boldsymbol{I}_{n} & 0 \\ \frac{1}{s_{k}}\mathcal{D} \Phi (\tilde{\boldsymbol{x}}^{k})&\boldsymbol{I}_{m} 
\end{array}\right).\label{eq88 }
\end{equation}
The matrix $\boldsymbol{\Sigma}_{k}$ is positive-definite for any $r_{k}, s_{k}>0$.
\begin{equation} \boldsymbol{\Sigma}_{k}=\boldsymbol{Q}_{k}\boldsymbol{M}_{k}^{-1}=\left(
\begin{array}{cc}
r_{k}\boldsymbol{I}_{n}+\frac{1}{s_{k}}\mathcal{D} \Phi (\tilde{\boldsymbol{x}}^{k})^{T}\mathcal{D} \Phi (\tilde{\boldsymbol{x}}^{k}) & -\mathcal{D} \Phi (\tilde{\boldsymbol{x}}^{k})^{T} \\
-\mathcal{D} \Phi (\tilde{\boldsymbol{x}}^{k})&s_{k}\boldsymbol{I}_{m} 
\end{array}
\right)\succ0.\label{eq89 }
\end{equation}
When the condition $(2-\beta)^{2}r_{k}s_{k}>\norm{\mathcal{D} \Phi (\tilde{\boldsymbol{x}}{}^{k})}^{2}$ holds, we have
\begin{equation} 
\boldsymbol{G}_{k}=\boldsymbol{Q}_{k}^{T}+\boldsymbol{Q}_{k}-\beta\boldsymbol{M}_{k}^{T}\boldsymbol{\Sigma}_{k}\boldsymbol{M}_{k}=\left(
\begin{array}{cc}
(2-\beta)r_{k}\boldsymbol{I}_{n} & -\mathcal{D} \Phi (\tilde{\boldsymbol{x}}^{k})^{T} \\ -\mathcal{D} \Phi (\tilde{\boldsymbol{x}}^{k})&(2-\beta)s_{k}\boldsymbol{I}_{m} 
\end{array}
\right)\succ0.\label{eq90 }
\end{equation}

\begin{algorithm}[t]\label{algorithm2}
\caption{Prediction correction method}
\LinesNumbered
\KwIn{ Parameters $\mu_{1}>1,\mu_{2}>1$; tolerant error $\tau$.}
\KwOut{The optimal solution: $\boldsymbol{x}^{*}$.}
Initialize $\boldsymbol{x}^{0},\boldsymbol{\lambda}^{0},k=0$\;
\While{$error\ge\tau\ $}{
\emph{\% The prediction step:}\\
$r_{k}=\frac{1}{\mu_{1}}\Vert\mathcal{D} \Phi (\boldsymbol{x}^{k})\Vert$\;
$\tilde{\boldsymbol{x}}^{k}=\arg\min\{\mathcal{L}''(\boldsymbol{x},\boldsymbol{\lambda}^{k})+\frac{r_{k}}{2}\norm{\boldsymbol{x}-\boldsymbol{x}^{k}}^{2}\mid \boldsymbol{x}\in \mathcal{X}\}$\;
$s_{k}=\frac{\mu_{2}}{r_{k}}\Vert\mathcal{D} \Phi (\tilde{\boldsymbol{x}}^{k})\Vert^{2}$\;
$\tilde{\boldsymbol{\lambda}}^{k}=P_{\mathcal{Z}}\left(\boldsymbol{\lambda}^{k}+\frac{1}{s_{k}}  \Phi (\tilde{\boldsymbol{x}}^{k}) \right)$\;
\emph{\% The correction step:}\\
$\beta^{*}={(\boldsymbol{w}^{k}-\tilde{\boldsymbol{w}}^{k})^{T}\boldsymbol{Q}_{k}(\boldsymbol{w}^{k}-\tilde{\boldsymbol{w}}^{k})}/{\Vert\boldsymbol{M}_{k}(\boldsymbol{w}^{k}-\tilde{\boldsymbol{w}}^{k})\Vert^{2}_{\boldsymbol{\Sigma}_{k}}}$\;
$\beta=\gamma\cdot\beta^{*}$, where $\gamma\in(0, (2-\sqrt{1/\mu})/\beta^{*})$\;
$\boldsymbol{w}^{k+1}=\boldsymbol{w}^{k}-\beta\boldsymbol{M}_{k}(\boldsymbol{w}^{k}-\tilde{\boldsymbol{w}}^{k})$\;
$error=\mbox{abs}[ f(\boldsymbol{x}^{k})- f(\boldsymbol{x}^{k+1})]$\;
$k=k+1$\;
} 
\Return { $\boldsymbol{x}^{*}=\boldsymbol{x}^{k}$}\;
 \end{algorithm}

\subsection{Step Size in the Correction Step}
In \cite{H1}, the authors investigate the selection of $\beta$ for a saddle-point problem with a linear coupling operator. The approach is to maximize the distance $\Vert\boldsymbol{w}^{k}-\boldsymbol{w}^{*}\Vert^{2}_{\boldsymbol{\Sigma}_{k}}-\Vert\boldsymbol{w}^{k+1}-\boldsymbol{w}^{*}\Vert^{2}_{\boldsymbol{\Sigma}_{k}}$. Specifically, to determine an approximate value of $\beta$, the distance is expressed as follows: 
\begin{equation} 
\vartheta(\beta)=\Vert\boldsymbol{w}^{k}-\boldsymbol{w}^{*}\Vert^{2}_{\boldsymbol{\Sigma}_{k}}-\Vert\boldsymbol{w}^{k+1}(\beta)-\boldsymbol{w}^{*}\Vert^{2}_{\boldsymbol{\Sigma}_{k}},\label{eq91 }
\end{equation}
where $\boldsymbol{w}^{k+1}(\beta)=\boldsymbol{w}^{k}-\beta\boldsymbol{M}(\boldsymbol{w}^{k}-\tilde{\boldsymbol{w}}^{k})$. To maximize $\vartheta(\beta)$, we cannot directly do so as the solution $\boldsymbol{w}^{*}$ is unknown. However, by setting $\boldsymbol{w}=\boldsymbol{w}^{*}$ in (\ref{eq69}), we get
\begin{equation} 
\begin{aligned}
(\tilde{\boldsymbol{w}}^{k}-\boldsymbol{w}^{*})^{T}\boldsymbol{Q}_{k}(\boldsymbol{w}^{k}-\tilde{\boldsymbol{w}}^{k})\ge&0\\
[(\boldsymbol{w}^{k}-\boldsymbol{w}^{*})-(\boldsymbol{w}^{k}-\tilde{\boldsymbol{w}}^{k})]^{T}\boldsymbol{Q}_{k}(\boldsymbol{w}^{k}-\tilde{\boldsymbol{w}}^{k})\ge&0\\
(\boldsymbol{w}^{k}-\boldsymbol{w}^{*})^{T}\boldsymbol{Q}_{k}(\boldsymbol{w}^{k}-\tilde{\boldsymbol{w}}^{k})\ge&(\boldsymbol{w}^{k}-\tilde{\boldsymbol{w}}^{k})^{T}\boldsymbol{Q}_{k}(\boldsymbol{w}^{k}-\tilde{\boldsymbol{w}}^{k}).
\end{aligned}\label{eq92}
\end{equation}
By the above inequality, we have 
\begin{equation} 
\begin{aligned}
\vartheta(\beta)=&\Vert\boldsymbol{w}^{k}-\boldsymbol{w}^{*}\Vert^{2}_{\boldsymbol{\Sigma}_{k}}-\Vert\boldsymbol{w}^{k}-\beta\boldsymbol{M}_{k}(\boldsymbol{w}^{k}-\tilde{\boldsymbol{w}}^{k})-\boldsymbol{w}^{*}\Vert^{2}_{\boldsymbol{\Sigma}_{k}}\\
=&2\beta(\boldsymbol{w}^{k}-\boldsymbol{w}^{*})^{T}\boldsymbol{Q}_{k}(\boldsymbol{w}^{k}-\tilde{\boldsymbol{w}}^{k})-\beta^{2}\Vert\boldsymbol{M}_{k}(\boldsymbol{w}^{k}-\tilde{\boldsymbol{w}}^{k})\Vert^{2}_{\boldsymbol{\Sigma}_{k}}\\
\stackrel{(\ref{eq92})}{\ge}&2\beta(\boldsymbol{w}^{k}-\tilde{\boldsymbol{w}}^{k})^{T}\boldsymbol{Q}_{k}(\boldsymbol{w}^{k}-\tilde{\boldsymbol{w}}^{k})-\beta^{2}\Vert\boldsymbol{M}_{k}(\boldsymbol{w}^{k}-\tilde{\boldsymbol{w}}^{k})\Vert^{2}_{\boldsymbol{\Sigma}_{k}}.
\end{aligned}\label{eq93 }
\end{equation}
We define the lower bound of distance as follows.
\begin{equation} 
\xi(\beta)=2\beta(\boldsymbol{w}^{k}-\tilde{\boldsymbol{w}}^{k})^{T}\boldsymbol{Q}_{k}(\boldsymbol{w}^{k}-\tilde{\boldsymbol{w}}^{k})-\beta^{2}\Vert\boldsymbol{M}_{k}(\boldsymbol{w}^{k}-\tilde{\boldsymbol{w}}^{k})\Vert^{2}_{\boldsymbol{\Sigma}_{k}}.\label{eq94 }
\end{equation}
Then $\xi(\beta)$ is a lower bound of $\vartheta(\beta)$. Hence, we maximize the value of $\xi(\beta)$ instead of $\vartheta(\beta)$. Notably, $\xi(\beta)$ is a quadratic function of $\beta$, and its maximum point is reached at
\begin{equation} 
\beta^{*}_{k}=\frac{(\boldsymbol{w}^{k}-\tilde{\boldsymbol{w}}^{k})^{T}\boldsymbol{Q}_{k}(\boldsymbol{w}^{k}-\tilde{\boldsymbol{w}}^{k})}{\Vert\boldsymbol{M}_{k}(\boldsymbol{w}^{k}-\tilde{\boldsymbol{w}}^{k})\Vert^{2}_{\boldsymbol{\Sigma}_{k}}}.\label{eq95 }
\end{equation}
Notice that the step size $\beta^{*}$ represents the extreme point of lower bound $\xi(\beta)$, rather than the actual distance difference $\vartheta(\beta)$. Therefore, for the $k$-th iteration, we set $\beta=\gamma\cdot\beta^{*}_{k}$ as the step size, where $\gamma$ is a tuning parameter. As highlighted in Section \ref{se42}, both corrective-matrix selections satisfy the condition $(2-\beta)^{2}r_{k}s_{k}>\norm{\mathcal{D} \Phi (\tilde{\boldsymbol{x}}{}^{k})}^{2}$. When the values of $(r_{k},s_{k})$ follow (\ref{eq65}) and (\ref{eq66}), it can be further simplified as $(2-\beta)^{2}\mu_{2}>1$. Thus, the tuning parameter should meet $0\leq \gamma\leq(2-\sqrt{1/\mu_{2}})/\beta^{*}$. The pseudo-code of the PC method is found in Algorithm \ref{algorithm2}.

\section{Numerical Experiments}\label{section5}
In this section, we evaluate the performance of the proposed methods.
We consider a general quadratically-constrained quadratic programming (QCQP) problem as follows.
\begin{equation}
\begin{aligned}
\min_{\boldsymbol{x}}\quad & \norm{\boldsymbol{A}\boldsymbol{x}-\boldsymbol{a}}^{2}\\
s.t. \quad&\norm{\boldsymbol{B}_{i}\boldsymbol{x}-\boldsymbol{b}_{i}}^{2}\leq c_{i},\forall \ i=1,\cdots,m,\label{eq96}
\end{aligned}
\end{equation}
where matrices $\boldsymbol{A},\boldsymbol{B}_{1},\cdots,\boldsymbol{B}_{m}\in \mathbb{R}^{n\times n}$, vectors $\boldsymbol{a},\boldsymbol{b}_{1},\cdots,\boldsymbol{b}_{m}\in \mathbb{R}^{n\times 1}$ and $\boldsymbol{c}=[c_{1},\cdots,c_{m}]^{T}\in \mathbb{R}^{m\times 1}$ are randomly generated.
The Lagrangian function can be written as 
\begin{equation}
\mathcal{L}'''(\boldsymbol{x},\boldsymbol{\lambda})=\norm{\boldsymbol{A}\boldsymbol{x}-\boldsymbol{a}}^{2}+\sum_{i=1}^{m} \lambda _{i}(\norm{\boldsymbol{B}_{i}\boldsymbol{x}-\boldsymbol{b}_{i}}^{2}- c_{i}).\label{eq97}
\end{equation}
The gradient operator is equal to 
\begin{equation} \mathcal{D} \Phi (\boldsymbol{x})=\left(
\begin{array}{c}
2(\boldsymbol{B}_{1}\boldsymbol{x} -\boldsymbol{b}_{1})^{T}\boldsymbol{B}_{1}\\
\vdots\\
2(\boldsymbol{B}_{m}\boldsymbol{x} -\boldsymbol{b}_{m})^{T}\boldsymbol{B}_{m}
\end{array}\right).\label{eq98 }
\end{equation}
\subsection{Solving (\ref{eq96}) by the Relaxed PPA }
In the customized PPA, for the $k$-th iteration, we utilize formula (\ref{eq65}) to generate the factor parameter $r_{k}$, followed by updating the following variables. 
\begin{equation}
\tilde{\boldsymbol{x}}^{k}=\arg\min\left\{\norm{\boldsymbol{A}\boldsymbol{x}-\boldsymbol{a}}^{2}+\sum_{i=1}^{m} \lambda_{i}^{k}\norm{\boldsymbol{B}_{i}\boldsymbol{x}-\boldsymbol{b}_{i}}^{2}+\frac{r_{k}}{2}\Vert\boldsymbol{x}-\boldsymbol{x}^{k}\Vert^{2}\mid\boldsymbol{x}\in \mathbb{R}^{n}\right\}.\label{eq99 }
\end{equation}
\begin{equation}
\tilde{\boldsymbol{\lambda}}^{k}=\arg\max\left\{\sum_{i=1}^{m} \lambda_{i}(\Vert\boldsymbol{B}_{i}\tilde{\boldsymbol{x}}{}^{k}-\boldsymbol{b}_{i}\Vert^{2}- c_{i})+\boldsymbol{\lambda}^{T}\mathcal{D} \Phi (\tilde{\boldsymbol{x}}^{k})(\tilde{\boldsymbol{x}}^{k}-\boldsymbol{x}^{k}) -\frac{s_{k}}{2}\Vert\boldsymbol{\lambda}-\boldsymbol{\lambda}^{k}\Vert^{2}\mid \boldsymbol{\lambda}\in \mathbb{R}^{m}_{+}\right\}.\label{eq100 }
\end{equation}
Further, they can be written as 
\begin{equation}
\tilde{\boldsymbol{x}}^{k}=\left(2\boldsymbol{A}^{T}\boldsymbol{A}+2\sum_{i=1}^{m} \lambda^{k}_{i}\boldsymbol{B}^{T}_{i}\boldsymbol{B}_{i}+r_{k}\boldsymbol{I}_{n}\right)^{-1}\cdot\left(2\boldsymbol{A}^{T}\boldsymbol{a}+2\sum_{i=1}^{m} \lambda^{k}_{i}\boldsymbol{B}^{T}_{i}\boldsymbol{b}_{i}+r_{k}\boldsymbol{x}^{k}\right).\label{eq101 }
\end{equation}
\begin{equation}
\tilde{\lambda}^{k}_{i}=\min\{\hat{\lambda}^{k}_{i},0\}, \mbox{ where } \hat{\lambda}^{k}_{i}= \lambda^{k}_{i} +\frac{1}{s_{k}}\left[\Vert\boldsymbol{B}_{i}\tilde{\boldsymbol{x}}{}^{k}-\boldsymbol{b}_{i}\Vert^{2}- c_{i}+2(\boldsymbol{B}_{i}\tilde{\boldsymbol{x}}^{k}-\boldsymbol{b}_{i})^{T}\boldsymbol{B}_{i}(\tilde{\boldsymbol{x}}^{k}-\boldsymbol{x}^{k})\right], \forall i.\label{eq102 }
\end{equation}
In the above equations, we update the factor parameter $s_{k}$ using formula (\ref{eq66}). We then relax the equations as follows.
\begin{equation}\boldsymbol{w}^{k+1}=\boldsymbol{w}^{k}-\gamma(\boldsymbol{w}^{k}-\tilde{\boldsymbol{w}}^{k}).\label{eq103 }\end{equation}
\subsection{Solving (\ref{eq96}) by the PC Method }
In the prediction step, for the $k$-th iteration, the factor parameter $r_{k}$ is generated using formula (\ref{eq65}). Following this, problem (\ref{eq97}) is solved to update the following variables. 
\begin{equation}
\tilde{\boldsymbol{x}}^{k}=\arg\min\left\{\norm{\boldsymbol{A}\boldsymbol{x}-\boldsymbol{a}}^{2}+\sum_{i=1}^{m} \lambda_{i}^{k}\norm{\boldsymbol{B}_{i}\boldsymbol{x}-\boldsymbol{b}_{i}}^{2}+\frac{r_{k}}{2}\Vert\boldsymbol{x}-\boldsymbol{x}^{k}\Vert^{2}\mid\boldsymbol{x}\in \mathbb{R}^{n}\right\}.\label{eq104 }
\end{equation}
\begin{equation}
\tilde{\boldsymbol{\lambda}}^{k}=\arg\max\left\{\sum_{i=1}^{m} \lambda_{i}(\Vert\boldsymbol{B}_{i}\tilde{\boldsymbol{x}}{}^{k}-\boldsymbol{b}_{i}\Vert^{2}- c_{i})-\frac{s_{k}}{2}\Vert\boldsymbol{\lambda}-\boldsymbol{\lambda}^{k}\Vert^{2}\mid \boldsymbol{\lambda}\in \mathbb{R}^{m}_{-}\right\}.\label{eq105 }
\end{equation}
Further, they can be written as 
\begin{equation}
\tilde{\boldsymbol{x}}^{k}=\left(2\boldsymbol{A}^{T}\boldsymbol{A}+2\sum_{i=1}^{m} \lambda^{k}_{i}\boldsymbol{B}^{T}_{i}\boldsymbol{B}_{i}+r_{k}\boldsymbol{I}_{n}\right)^{-1}\cdot\left(2\boldsymbol{A}^{T}\boldsymbol{a}+2\sum_{i=1}^{m} \lambda^{k}_{i}\boldsymbol{B}^{T}_{i}\boldsymbol{b}_{i}+r_{k}\boldsymbol{x}^{k}\right).\label{eq106 }
\end{equation}
\begin{equation}
\tilde{\lambda}^{k}_{i}=\min\{\hat{\lambda}^{k}_{i},0\}, \mbox{ where } \hat{\lambda}^{k}_{i}= \lambda^{k}_{i}+\frac{1}{s_{k}}\left(\Vert\boldsymbol{B}_{i}\tilde{\boldsymbol{x}}{}^{k}-\boldsymbol{b}_{i}\Vert^{2}- c_{i}\right), \forall i.\label{eq107 }
\end{equation}
In the above equation, the factor parameter $s_{k}$ is updated by formula (\ref{eq66}). In the correction step, we correct the predictive variables.
\begin{equation}\boldsymbol{w}^{k+1}=\boldsymbol{w}^{k}-\beta\boldsymbol{M}_{k}(\boldsymbol{w}^{k}-\tilde{\boldsymbol{w}}^{k}),\label{eq108 }\end{equation}
where we have two choices for the corrective matrix, such as
\begin{equation} \boldsymbol{M}_{k}=\left(
\begin{array}{cc}
\boldsymbol{I}_{n} & 0 \\ \frac{1}{s_{k}}\mathcal{D} \Phi (\tilde{\boldsymbol{x}}^{k})&\boldsymbol{I}_{m} 
\end{array}\right) \mbox{ or } \left(
\begin{array}{cc}
\boldsymbol{I}_{n} & \frac{-1}{r_{k}}\mathcal{D} \Phi (\tilde{\boldsymbol{x}}^{k})^{T} \\0&\boldsymbol{I}_{m} 
\end{array}\right).\label{eq109 }
\end{equation}

\subsection{Numerical Results }
In the experiment, matrices and vectors are randomly generated in Matlab by using $\mathsf{rand(m,n)}$. For the constraints, the vector $\boldsymbol{z}$ is generated by using $10(1+\mathsf{rand(m,1)})$. The initial values of primal and dual variables are set to $\boldsymbol{x}^{0}=(\boldsymbol{A}^{T}\boldsymbol{A})^{-1}\boldsymbol{A}^{T}\boldsymbol{a}$ and $\boldsymbol{\lambda}^{0}=\mathsf{-ones(m,1)}$, respectively. Two factor parameters $\mu_{1}$ and $\mu_{2}$ are equal to 9 and 1.2, respectively. The stopping error is equal to $\tau =10^{-10}$. In addition, all proposed methods are implemented in a serial way, rather than a parallel way. In the following tables and figures, the numerical performance of the tested methods is evaluated by the following parameters:
\begin{itemize}
\item Iter: required iteration number;
\item Time: total computing time in seconds;
\item CPU: maximum total running time of subproblems in seconds.
\item Error: function difference between adjacent iterations. 
\end{itemize}
\begin{figure}[t]
\centering
\begin{minipage}[t]{0.5\linewidth}
\centering
\includegraphics[width=3.2in]{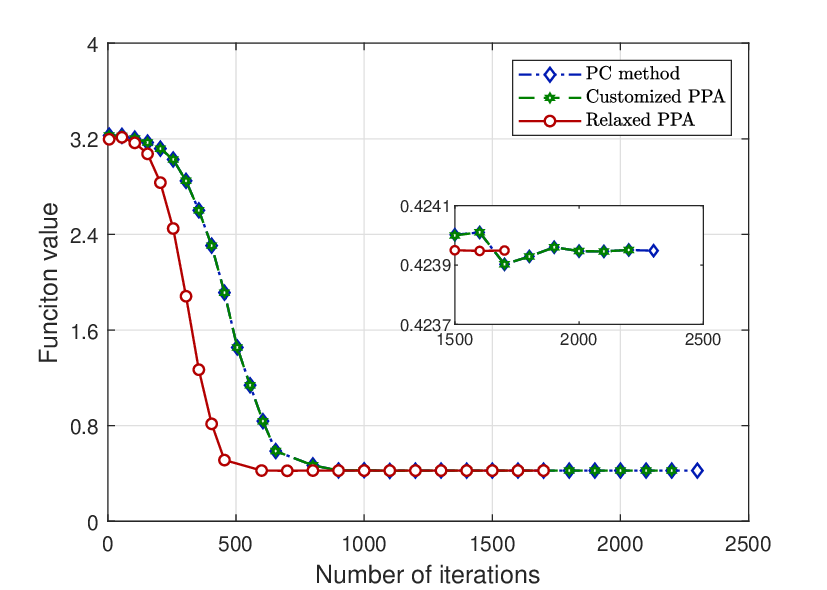}
\caption{Function value versus iterations}
  \label{fig:1}
\end{minipage}%
\begin{minipage}[t]{0.5\linewidth}
\centering
\includegraphics[width=3.2in]{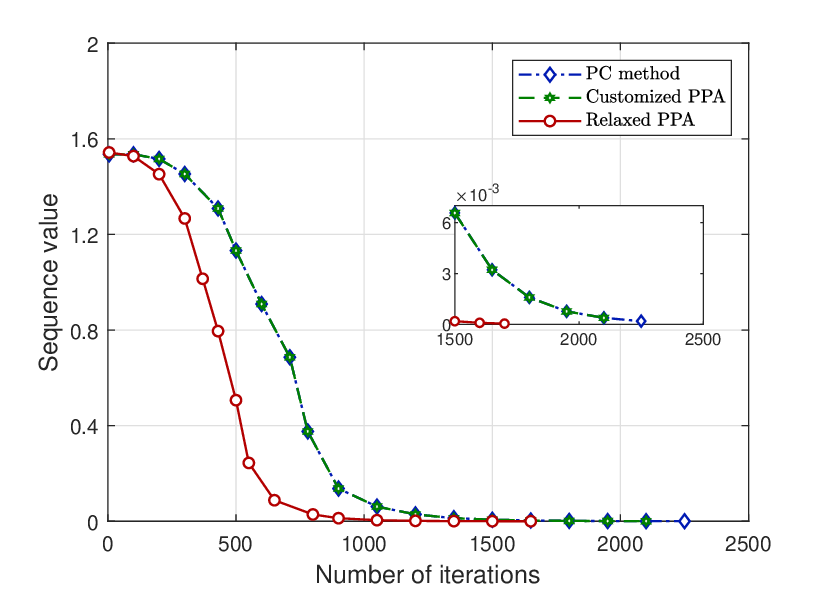}
\caption{Sequence value versus iterations}
  \label{fig:2}
\end{minipage}
\end{figure}
Our main interest is to show the convergence behavior of the proposed methods. In Figure \ref{fig:1}, we show the function value versus the number of iterations. It can be observed that the relaxed PPA requires the least number of iterations for the given stopping error. The other methods have a similar shape, whereas the PC method requires the most iterations. Additionally, we notice that the function value converges with a slight oscillation, which is due to the initial sequence being infeasible. This implies that the infeasible sequence converges to the feasible domain. Figure \ref{fig:2} illustrates the sequence value versus iterations, where the sequence value is defined as $\Vert\boldsymbol{x}^{k}-\boldsymbol{x}^{*}\Vert$. We can observe that the relaxed PPA converges to the optimal solution  and these sequences of proposed methods are strictly contracted. In Table \ref{table1}, we compare their performances by setting different dimensions. It can be observed that the relaxed PPA requires the least iterations, and it takes the longest time for the large-scale problem due to the large CPU. The performance of the customized PPA is comparable to that of the PC method. Furthermore, we can also observe that the errors of the proposed methods are smaller than the stopping error.

\setlength{\aboverulesep}{0pt}
\setlength{\belowrulesep}{0pt}
\begin{table*}[thb]\centering
\caption{The performances versus proposed methods} 
\scalebox{0.8}{
\begin{tabular}{@{}rccrrrrrrrrrrrrrr@{}}\toprule
%\rowcolor{gray}
\multicolumn{2}{c}{Dimension} & \phantom{.}& \multicolumn{4}{c}{Customized PPA $(\gamma=1.0)$} & \phantom{.} & \multicolumn{4}{c}{Relaxed PPA $(\gamma=1.5)$}& \phantom{.} & \multicolumn{4}{c}{PC Method $(\gamma=1.0)$}\\ \cmidrule{1-2} \cmidrule{4-7} \cmidrule{9-12} \cmidrule{14-17}%$\vv{\text{Number of MDs}}$ 
 $\quad m$& $n$ && Iter&Time & CPU &Error&& Iter&Time & CPU &Error&&Iter&Time & CPU &Error\\ \midrule
10 &30 && 1318& 0.99& 0.012& 4.86e-12&
      &1209&0.98& 0.011& 1.90e-11&
      &1819&1.26& 0.0104& 5.73e-11\\
10 &60 && 2318&5.74& 0.020&9.33e-11&
	   &1674&5.14& 0.018& 5.82e-12&
	   &2529&6.21&0.019&2.38e-11\\
10 &90 && 13547&65.04& 0.021&9.99e-11& 
	   &9215&137.73& 0.041& 9.98e-11&
	   &13547&66.06&0.021&9.99e-11\\
20 &30 && 1765&2.02& 0.019&4.37e-11& 
		&1452&2.85& 0.014& 9.74e-11&
		&1766&2.02&0.0108&4.25e-11\\
20 &60 && 3644&15.23& 0.027&6.33e-11&
		 &2423&19.07& 0.020& 1.41e-11&
		 &3644&14.78&0.020&6.06e-11\\
20 &90 && 20688&159.98& 0.022&9.99e-11&
		 &14068&553.09& 0.071& 9.98e-11&
		 &20688&163.47&0.041&9.99e-11\\
30 &30 && 2205&3.32& 0.013&3.76e-11&
		 & 1790&7.29& 0.016& 5.17e-11&
		 &2367&3.65&0.017&9.62e-11\\
30 &60 && 4053& 22.55& 0.024&1.92e-11& 
		&2691&29.79& 0.028& 5.67e-11&
		&3842&21.56&0.027&2.95e-12\\
30 &90 && 22175&240.86& 0.024&9.99e-11&
		 &15059&917.59& 0.127& 9.99e-11&
		 &22175&241.91&0.027&9.99e-11\\
\bottomrule
\end{tabular}}
\label{table1}
\end{table*}

\section{Conclusion}\label{section6}
Nonlinear convex problems are frequently encountered in applied mathematics and engineering applications. To tackle these problems, this paper develops PPA-like methods, namely the relaxed PPA and PC method. The convergence of these methods can be directly obtained through the unified contraction framework. The theoretical analysis reveals that the proximal matrix varies with the generated sequence. In the case of the relaxed PPA, a customized symmetrical positive-definite proximal matrix is required. Therefore, two regularization parameters need to be adjusted at each iteration. Similarly, the predictive and corrective matrices in the PC method were redesigned. In addition, we have theoretically proven that the two algorithms can achieve a convergence rate of $O(1/t)$ by constructing a custom monotonically decreasing sequence $\{r_{k},s_{k}\}$, although this sequence is not easy to construct. Numerical results confirm the validity of the theoretical analysis.

\section*{Acknowledgments}
This work was supported by the National Natural Science Foundation of China (No. 62071212) and Guangdong Basic and Applied Basic Research Foundation (No. 2019B1515130003).

\bibliographystyle{elsarticle-num}
\bibliography{ref}

\end{document}